\documentclass[12pt]{amsart}
\usepackage{amsmath}
\usepackage{amscd}
\usepackage{amssymb}
\usepackage{amsfonts}

\newtheorem{theorem}{Theorem}[section]
\newtheorem{lemma}[theorem]{Lemma}
\newtheorem{corollary}[theorem]{Corollary}
\newtheorem{proposition}[theorem]{Proposition}

\theoremstyle{definition}

\theoremstyle{remark}
\newtheorem{remark}[theorem]{Remark}
\numberwithin{equation}{section}

\begin{document}
\title[Four-dimensional shrinkers with nonnegative Ricci curvature]
{Four-dimensional shrinkers with nonnegative Ricci curvature}

\author{Guoqiang Wu}
\address{School of Science, Zhejiang Sci-Tech University, Hangzhou 310018, China}
\email{gqwu@zstu.edu.cn}

\author{Jia-Yong Wu}
\address{Department of Mathematics and Newtouch Center for Mathematics,
Shanghai University, Shanghai 200444, China}
\email{wujiayong@shu.edu.cn}
\thanks{}
\dedicatory{}
\date{\today}

\begin{abstract}
In this paper, we investigate classifications of $4$-dimensional simply
connected complete noncompact nonflat shrinkers satisfying
$Ric+\mathrm{Hess}\,f=\tfrac 12g$ with nonnegative Ricci curvature. One one hand,
we show that if the sectional curvature $K\le 1/4$ or the sum of smallest two
eigenvalues of Ricci curvature has a suitable lower bound, then the shrinker is
isometric to $\mathbb{R}\times\mathbb{S}^3$. We also show that if the scalar
curvature $R\le 3$ and the shrinker is asymptotic to $\mathbb{R}\times\mathbb{S}^3$,
then the Euler characteristic $\chi(M)\geq 0$ and equality holds if and only if
the shrinker is isometric to $\mathbb{R}\times\mathbb{S}^3$. On the other hand,
we prove that if $K\le 1/2$ (or the bi-Ricci curvature is nonnegative) and
$R\le\tfrac{3}{2}-\delta$ for some $\delta\in (0,\tfrac{1}{2}]$, then the
shrinker is isometric to $\mathbb{R}^2\times\mathbb{S}^2$. The proof of
these classifications mainly depends on the asymptotic analysis by the evolution
of eigenvalues of Ricci curvature, the Gauss-Bonnet-Chern formula with boundary
and the integration by parts.
\end{abstract}
\maketitle

\section{Introduction}

Let $(M,g)$ be an $n$-dimensional complete Riemannian manifold. If there
exists a smooth potential function $f$ on $(M,g)$ such that
\begin{align}\label{Eq1}
Ric+\mathrm{Hess}\,f=\tfrac 12g,
\end{align}
where $Ric$ is the Ricci curvature of $(M,g)$ and $\text{Hess}\,f$ is the
Hessian of $f$, then the triple $(M, g, f)$ is called a \emph{gradient
shrinking Ricci soliton} or \emph{shrinker} (see \cite{[Ham]}). Shrinkers
are natural generalizations of Einstein manifolds and they arise as limits
of dilations of Type I singularities in the Ricci flow \cite{[Ham]}. They
can also be regarded as critical points of the Perelman's entropy functional
and play a significant role in Perelman's resolution of the Poincar\'e
conjecture \cite{[Pe],[Pe2],[Pe3]}. Therefore, it would be very important
to understand and even classify shrinkers.

For dimension $2$, Hamilton \cite{[Ham]} completely classified shrinkers.
For dimension $3$, following \cite{[CCZ],Iv,[NW],[Pe2]}, the classification
is also complete. For dimension $4$ (or higher), the complete classification
remains open but there has been a lot of progress under various curvature
assumptions. For example, Hamilton \cite{H86} gave a classification of
$4$-dimensional compact shrinkers with positive curvature operator. H. Chen
\cite{Chh} later got the same classification under the weaker curvature
assumption of $2$-positive curvature operator (the higher case due to \cite{BW}).
In the noncompact case, Naber \cite{[Na]} classified $4$-dimensional
noncompact shrinkers with bounded and nonnegative curvature operator.
In \cite{[MW]}, Munteanu-Wang removed Naber's bounded curvature assumption
and classified shrinkers of dimension $n\ge4$ with nonnegative curvature operator.
Recently there are many interesting classifications under special curvature
conditions, such as locally conformally flat, half-conformally flat,
harmonic Weyl curvature, half harmonic Weyl curvature, Bach-flat, constant scalar
curvature, positive isotropic curvature, half nonnegative isotropic curvature,
Weyl curvature pinching, etc., see \cite{CaCh,CJZ,CaXi,CWZ,Cat,ChWy,Cheng-Zhou,ENM,[FIK],FrR,[FR],[GLX],LNW,[MS],[NW],
[PW],WWW,Zhz} and the references therein. For the K\"ahler case, the reader can refer to \cite{CaJT,Chx,[Chetc],GuZh,[LN],[Ni],[WZ],[Zh]} and the references therein. Besides, there
are some important works to figure out the rigidity of cylinders in the moduli space
of shrinkers,  see, e.g., \cite{CoMi,LiW19,LW24} and the references therein.

As we currently known, every nontrivial complete noncompact $4$-dimensional
shrinker either splits locally as a product or has a single end smoothly
asymptotic to a cone. It is naturally conjectured that any noncompact
$4$-dimensional shrinker must fit one of the two above, at least asymptotically.
Based on this, Munteanu-Wang \cite{[MWa],[MW17],[MW19]} proposed an interestingly
dichotomy in terms of the scalar curvature $R$. If $\lim_{x\to\infty}R=0$ , they
showed that each end of $(M^4,g)$ must be smoothly asymptotic to a cone. If $R$
is bounded and $R\ge c>0$, they showed that either $\lim_{x\to\infty}R=\frac 32$
and each end of $(M^4,g)$ is smoothly asymptotic to a quotient of
$\mathbb{R}\times\mathbb{S}^3$ or $\lim_{x\to\infty}R=1$ and the sequence
of pointed manifolds $(M^4,g,x_i)$ subconverges to a quotient of
$\mathbb{R}^2\times\mathbb{S}^2$ for any sequence of points $x_i$ going to
infinity along an integral curve of $\nabla f$. (See also \cite{[Na]} for
a similar version under stronger curvature assumptions by a different argument.)
When $(M^4,g)$ is K\"ahler, the above proposed dichotomy indeed occurs and
it has been completely classified recently by combining many efforts
\cite{BCCD,CCD,CDS,[LW2023]} and the references therein. When $(M^4,g)$
is Riemannian, it remains open but there has been some progress on this
direction. For example, Kotschwar-Wang \cite{KWa15} showed that two shrinkers
$C^2$-asymptotic to the same cone must be isometric. So the classification
issue of asymptotically conical shrinkers is reduced to that of asymptotic cones.
Indeed, Kotschwar-Wang \cite{KWa24} further showed that complete shrinker which agrees
to infinite order at spatial infinity with a standard cylinder along some end
must be isometric to the cylinder on that end.

In another direction, Munteanu-Wang \cite{[MW]} showed that any $n$-dimensional
shrinker with positive Ricci curvature and nonnegative sectional curvature must
be compact, which answers a question of Cao \cite{[Cao]}. This compact
property can be extended to the positive $2$nd-Ricci curvature case \cite{[LN]}
and various curvature-pinching conditions \cite{[WW]}. H.-D. Cao indeed
proposed a question that any shrinker with positive Ricci curvature
is compact. This question is open in dimension $4$ and higher. More generally,
it is expected that any $4$-dimensional simply connected complete noncompact
nonflat shrinker with nonnegative Ricci curvature should be isometric to
$\mathbb{R}\times \mathbb{S}^3$ or $\mathbb{R}^2\times \mathbb{S}^2$.

Based on Cao's problem, in this paper, we are concerned with some classifications
of $4$-dimensional shrinkers with nonnegative Ricci curvature under various
additional conditions. We actually focus on the asymptotic limit confirmed by
Munteanu-Wang \cite{[MW19]} completely determining shrinkers by choosing suitable
curvature or topological conditions. In the whole paper, let $\lambda_1\le \lambda_2\le\lambda_3\le\lambda_4$ denote eigenvalues of  Ricci curvature in
$(M,g,f)$ and let $\{e_i\}_{i=1}^4$ denote the corresponding eigenvectors.
We state main results of this paper as follows.

%We also use $D(a)$ to denote the sublevel set of $f$, $\Sigma(a)$ to denote the level set of $f$, i.e.
%$D(a)=\{f\leq a\}$, $\Sigma(a)=\{f=a\}$.

%In the paper, we study four dimensional nonflat noncompact Ricci shrinker with nonnegative Ricci shrinker. Naber \cite{[Na]} obtained the asymptotic limit is $\mathbb{R}\times\mathbb{S}^3$ or $\mathbb{R}^2\times\mathbb{S}^2$, and it is expected Ricci shrinker with nonnegative Ricci curvature should be isometric to $\mathbb{R}\times\mathbb{S}^3$ or $\mathbb{R}^2\times\mathbb{S}^2$.

\subsection{Rigidity when the asymptotic limit is $\mathbb{R}\times\mathbb{S}^3$}

In this subsection we mainly classify shrinkers when asymptotic limit
of shrinker is $\mathbb{R}\times\mathbb{S}^3$. First, we consider the case when
the sectional curvature has an upper bound.
\begin{theorem}\label{clasfia}
Let $(M,g, f)$ be a $4$-dimensional simply connected complete noncompact
nonflat shrinker with the Ricci curvature $Ric\ge0$. If the sectional curvature
$K\le 1/4$, then $(M, g, f)$ must be isometric to
$\mathbb{R}\times\mathbb{S}^3$.
\end{theorem}
In Theorem \ref{clasfia}, the upper bound of sectional curvature is sharp,
since $\mathbb{R}\times\mathbb{S}^3$ has exact sectional curvature bound $1/4$.
The theorem proof combines the Munteanu-Wang's argument \cite{[MW],[MW19]}
and Naber's structure theorem \cite{[Na]} (an improved version in \cite{[MW19]}).
The proof idea is as follows. If the shrinker is noncompact, by Naber's structure
theorem, there exists a sequence of pointed manifolds smoothly converges to
a limit manifold along each integral curve of $\nabla f$. Then, we study the
scalar curvature of each limit manifold by analyzing the integral curve of
$\nabla f$ and the $f$-Laplacian equation of the first eigenvalue of the
Ricci curvature. This process is involved in many delicate estimates and
integration by parts on manifolds with boundary. We are able to prove that
the integral of $f$-Laplacian evolution of the first eigenvalue indeed is
an equality. Thus, this forces the scalar curvature $R\equiv3/2$ outside
of a compact set of shrinker. Finally the conclusion follows by the real
analyticity of solitons and the complete classification of 4-dimensional
shrinkers with constant scalar curvature due to \cite{Cheng-Zhou,[FR],[PW2]}.

In the same spirit of proving Theorem \ref{clasfia}, we can apply the
asymptotic analysis method to prove another rigid characterization of
$\mathbb{R}\times \mathbb{S}^3$ under suitable assumption of the sum
of the smallest two eigenvalues of Ricci curvature.

\begin{theorem}\label{sumeig}
Let $(M,g, f)$ be a $4$-dimensional simply connected complete noncompact nonflat
shrinker with $Ric\ge0$. If $\lambda_1+\lambda_2\ge\tfrac{R}{3}$ and $R$ is bounded,
then $(M, g, f)$ must be isometric to $\mathbb{R}\times \mathbb{S}^3$.
\end{theorem}

In Theorem \ref{sumeig}, the curvature assumption $\lambda_1+\lambda_2\ge\tfrac{R}{3}$
is sharp, since $\lambda_1+\lambda_2=\frac{R}{3}$ on $\mathbb{R}\times \mathbb{S}^3$.
If we give an accurate upper bound of the scalar curvature, we still obtain a
precise characterization of $\mathbb{R}\times\mathbb{S}^3$ by the geometric model
of the shrinker at infinity when the Euler characteristic vanishes.

\begin{theorem}\label{clasfi0}
Let $(M, g, f)$ be a $4$-dimensional simply connected complete noncompact
shrinker with $Ric\ge0$. If $R\le 3$ and $(M, g)$ is asymptotic to
$\mathbb{R}\times\mathbb{S}^3$,  then the Euler characteristic $\chi(M)\ge 0$
and equality holds if and only if the shrinker is isometric to
$\mathbb{R}\times\mathbb{S}^3$.
\end{theorem}

In Theorem \ref{clasfi0}, $(M, g, f)$ is asymptotic to $\mathbb{R}\times\mathbb{S}^3$
means that the sequence of pointed manifolds $(M,g,x_i)$ subconverges to
$\mathbb{R}\times\mathbb{S}^3$ for any sequence of points $x_i$ going to
infinity along an integral curve of $\nabla f$. The proof of Theorem
\ref{clasfi0} is similar to the proof strategy of Theorem \ref{clasfia}.
The biggest difference is that in the proof of Theorem \ref{clasfi0} we need
to apply the $4$-dimensional Gauss-Bonnet-Chern formula with boundary to control
the Weyl curvature by the Euler characteristic. Moreover, we are able to efficiently
estimate all boundary terms about curvature integrations on a large level set.
We would like to remark that more rigid results related to the Weyl curvature or the Euler
characteristic refer to Corollary \ref{Weylcon} and Theorem \ref{shrcomp}
in Section \ref{classifica}.

%Next we state a corollary.
%\begin{corollary} Suppose $(M^4, g, f)$ is a four dimensional simply connected Ricci shrinker with $Ric\geq 0$. If $R\leq \frac{3}{2}$ and
%\begin{align*}
%\int_{M}|W|^2 \leq 32\pi^2,
%\end{align*}
%then $(M^4, g, f)$ is isometric to $\mathbb{R}\times \mathbb{S}^3$.
%\end{corollary}

\subsection{Rigidity when the asymptotic limit is $\mathbb{R}^2\times\mathbb{S}^2$}

In this subsection we shall classify shrinkers with nonnegative Ricci
curvature when the asymptotic limit is $\mathbb{R}^2\times\mathbb{S}^2$.
\begin{theorem}\label{clasfi}
Let $(M,g, f)$ be a $4$-dimensional simply connected complete noncompact
nonflat shrinker with $Ric\ge0$. If $K\le 1/2$ and $R\leq \frac{3}{2}-\delta$
for some $\delta\in (0, \frac{1}{2}]$, then $(M, g, f)$ must be isometric to
$\mathbb{R}^2\times\mathbb{S}^2$.
\end{theorem}

In Theorem \ref{clasfi}, the upper bound of sectional curvature is sharp, since $\mathbb{R}^2\times\mathbb{S}^2$ has constant sectional curvature $1/2$. The
proof strategy of Theorem \ref{clasfi} is similar to the argument of Theorem
\ref{clasfia}. But the most difference is that we consider the $f$-Laplacian
equation of the sum  of the smallest two eigenvalues of Ricci curvature,
which involves delicate analysis of a troublesome boundary term. Luckily,
we apply the level set method to successfully estimate this term by combining
Zhu's volume estimate \cite{[Zhubo]}.

Besides, adopting a similar argument of Theorem \ref{clasfi}, we can apply the nonnegative
bi-Ricci curvature instead of the sectional curvature condition to characterize $\mathbb{R}^2\times\mathbb{S}^2$ for the shrinker. For any pair of unit orthonormal
vectors $u$ and $v$ on $(M,g)$, the bi-Ricci curvature introduced by Shen-Ye \cite{ShYe}
is defined by
\[
BiRic(u, v)=Ric(u, u)+Ric(v, v)-K(u, v),
\]
where $K(u, v)$ is the sectional curvature of plane spanned by $u$ and $v$.
This curvature lies somewhere between Ricci curvature and scalar curvature. From
its definition, the bi-Ricci curvature is the sum of the sectional curvatures
over all mutually orthogonal $2$-planes containing at least one of these tangent
vectors. Hence, the nonnegativity of the sectional curvature implies the
nonnegativity of the bi-Ricci curvature. The inverse is not true.

\begin{theorem}\label{bi-Ricci}
Let $(M,g, f)$ be a $4$-dimensional simply connected complete noncompact
nonflat shrinker with $Ric\ge0$. If the bi-Ricci curvature is nonnegative and
$R\leq \frac{3}{2}-\delta$ for some $\delta\in (0,\frac 12)$, then $(M, g, f)$
must be isometric to $\mathbb{R}^2\times \mathbb{S}^2$.
\end{theorem}

To sum up, our proof of all above theorems mainly depends on the delicate asymptotic
analysis of the $f$-Laplacian evolution of eigenvalues of the Ricci curvature.
The argument of integration by part is also often used. In particular, some
technique seems to be first appeared in the Ricci soliton theory, such as the
the $4$-dimensional Gauss-Bonnet-Chern formula with boundary. In last section
of this paper, we observe that our proof strategy is also suitable to Cheng-Zhou's
classification of $4$-dimensional noncompact shrinkers with $R=1$ in \cite{Cheng-Zhou}.

The paper is organized as follows. In Section \ref{pre}, we give some basic
facts and known results often used in the proof of our theorems. In Section
\ref{classifica}, we prove the some rigid results when the asymptotic limit
is $\mathbb{R}\times \mathbb{S}^3$. In particular, we will prove Theorems
\ref{clasfia}, \ref{sumeig} and \ref{clasfi0}. In Section \ref{R2S2}, we discuss
some rigid results when the asymptotic limit is $\mathbb{R}^2\times \mathbb{S}^2$.
We shall prove Theorems  \ref{clasfi} and \ref{bi-Ricci}. In Section \ref{csc},
we give an alternative proof of Cheng-Zhou's result in \cite{Cheng-Zhou} by
the asymptotic analysis via the integration method.

%\subsection{Rigidity when scalar curvature $R=1$}
%Recall that in Petersen and Wylie's paper \cite{[PW2]}, a gradient Ricci soliton $(M, g)$ is said to be %rigid if it is isometric to a quotient $N \times \mathbb{R}^k$, the product soliton of an Einstein %manifold $N$ of positive scalar curvature with the Gaussian soliton $\mathbb{R}^k$.
%A gradient Ricci soliton has constant scalar curvature if it is rigid.
%Conversely, Prof. Huai-Dong Cao conjectured that a gradient Ricci soliton is also rigid if has constant %scalar curvature.

%Cheng-Zhou \cite{Cheng-Zhou} proved a four dimensional Ricci shrinker is rigid if $R=1$, hence gave a %complete classification in dimension 4.
%In \cite{Ou-Qu-Wu}, the authors gave a new proof of Cheng-Zhou's result,  and the new proof is inspired by %the work of Petensen-Wylie \cite{[PW2]}. Next we will provide a new proof based on integral method. %Compared to the proof of Theorem \ref{theorem1.5},
%in the setting of $R=1$, we can control the most difficult term $K_{12}$ effectively.
 % We restate the main result as follows.
%\begin{theorem}[\cite{Cheng-Zhou}] Suppose $(M^4, g, f)$ is a four dimensional simply connected shrinking %gradient Ricci soliton with $R=1$, then $(M^4, g, f)$ is isometric to $\mathbb{R}^2\times \mathbb{S}^2$.
%\end{theorem}

\section{Preliminaries}\label{pre}
In this section, we mainly collect some known facts about shrinkers.
These results will be used in the proof of our theorems. We start to
recall some basic identities of the shrinker, which can be referred
to \cite{[Ha],[Ham]}. From \eqref{Eq1}, we trace it and get
\begin{equation}\label{tra}
R+\Delta f=\tfrac n2.
\end{equation}
From \eqref{Eq1} and \eqref{tra}, by adding a constant to $f$ if necessary,
we get
\begin{equation}\label{equat}
R+|\nabla f|^2=f
\end{equation}
and $\nabla R=2Ric(\nabla f)$. Thus, we have
\begin{equation}\label{feq}
\Delta_f f=\tfrac n2-f,
\end{equation}
where $\Delta_f:=\Delta-\nabla f\cdot\nabla$ is the weighted Laplacian.
Meanwhile, the scalar curvature and the Ricci curvature satisfy
\begin{equation}\label{Sequat}
\Delta_f R= R-2|Ric|^2
\end{equation}
and
\begin{equation}\label{evo}
\Delta_fR_{ij}=R_{ij}-2R_{ikjl}R_{kl},
\end{equation}
where $R_{ij}$ and $R_{ikjl}$ are the Ricci curvature and the Riemannian
curvature of $(M,g,f)$.

For $4$-dimensional shrinker $(M,g,f)$ with $Ric\ge0$, we let $0\le\lambda_1\le\lambda_2\le\lambda_3\le\lambda_4$ be eigenvalues
of Ricci curvature. At a point $x\in M$, let $e_1$ be
an eigenvector with respect to the minimal eigenvalues $\lambda_1$.
Then we extend $e_1$ to an orthonormal basis $\{e_1,e_2,e_3, e_4\}$
such that $\{e_i\}^4_{i=1}$ are the eigenvectors of $Ric$ with respect
to the corresponding eigenvalues $\{\lambda_i\}^4_{i=1}$. From \eqref{evo},
we have
\begin{equation}\label{evoRic0}
\Delta_f \lambda_1\le \lambda_1-2Rm(e_1,e_i,e_1,e_j)Ric(e_i,e_j)
\end{equation}
in the barrier sense. Diagonalizing the Ricci curvature such that
$R_{kl}:=Ric(e_k,e_l)=\lambda_k\delta_{kl}$, then
\[
2Rm(e_1,e_i,e_1,e_j)Ric(e_i,e_j)=2(K_{12}\lambda_2+K_{13}\lambda_3+K_{14}\lambda_4),
\]
where $K_{ij}$ denotes the sectional curvature of the plane spanned by $e_i$
and $e_j$. Thus \eqref{evoRic0} is written as
\begin{equation}\label{evoRia1}
\Delta_f \lambda_1\le \lambda_1-2(K_{12}\lambda_2+K_{13}\lambda_3+K_{14}\lambda_4)
\end{equation}
in the barrier sense. In the following section, we will repeatedly apply
the inequality \eqref{evoRia1} to prove our results.

Besides, the scalar curvature of shrinker $R\ge 0$ (see \cite{[Chen]})
and from \cite{[PiRS]}, $R>0$ unless $(M,g,f)$ is the Gaussian shrinker
$(\mathbb{R}^n, g_E, \frac{|x|^2}{4})$. On non-flat shrinker, Chow-Lu-Yang
\cite{[CLY]} refined scalar curvature estimate to be a sharp form:
\[
R\ge c\,f^{-1}
\]
for some constant $c>0$, and the equality holds at some K\"ahler shrinker
constructed by Feldman-Ilmanen-Knopf \cite{[FIK]}. By \cite{[CZ],[Chetc],[HM]},
the potential function $f$ has the following sharp estimate
\begin{theorem}\label{pote}
Let $(M,g, f)$ be an $n$-dimensional complete noncompact shrinker
with \eqref{Eq1} and \eqref{equat}. Then there exists a point $p\in M $
where $f$ attains its infimum. Moreover,
\[
\tfrac 14\left[\big(r(x)-5n\big)_{+}\right]^2\le f(x)\le\tfrac 14\left(r(x)+\sqrt{2n}\right)^2,
\]
where $r(x)$ is a distance function from $p$
to $x$, and $a_+=\max\{a,\,0\}$ for $a\in \mathbb{R}$.
\end{theorem}

Combining Theorem \ref{pote} with \eqref{equat}, we see that if $R$ is bounded,
then there exists $a>0$ such that the level set
\[
\Sigma(r):=\{x\in M|f(x)=r\}
\]
of $f$ is a compact manifold for $r\ge a$. Meanwhile, the
domain
\[
D(r):=\{x\in M|f(x)\le r\}
\]
is also a compact manifold with boundary $\Sigma(r)$.

From \cite{[CZ],[MW12]}, we know that the volume of the geodesic is
at least the linear growth and at most the Euclidean growth on an
$n$-dimensional complete noncompact shrinker $(M,g,f)$. That is,
there exist constants $c_1$ and $c_2$ such that
\[
c_1 r\le Vol(B_p(r))\le c_2r^n
\]
for any $r>1$, where $Vol(B_p(r))$ denotes the volume of the geodesic
ball $B_p(r)$ with radius $r$ and center at $p\in M$. Using this and
Theorem \ref{pote}, we have the weighted volume of $M$ is finite,
i.e., $V_f(M):=\int_Me^{-f}dv<\infty$.

It is well-known that the curvature operator on any $3$-dimensional shrinker
must be nonnegative, and hence is bounded by the scalar curvature. For
dimension $4$ and higher, the curvature operator may be has a mixed sign;
see \cite{[FIK]} for the detailed explanation. But for dimension $4$,
Munteanu-Wang \cite{[MWa]} proved that the curvature operator still could be
estimated by the scalar curvature if the scalar curvature is bounded. That
is,
\begin{theorem}\label{RmdyS}
Let $(M, g, f)$ be a $4$-dimensional shrinker with scalar curvature
$S\le c_1$ for some constant $c_1>0$. Then there exists a constant $c>0$
such that
\[
|Rm|\le c\,R
\]
on $(M, g, f)$, where $c$ depends only on $c_1$ and the geometry
of geodesic ball $B_p(r_0)$. Here $p$ is a minimum point of $f$ and
$r_0$ is determined by $c_1$.
\end{theorem}

The geometry structure of a shrinker can be studied by the associated Ricci flows
with such shrinker. By this method, Naber \cite{[Na]} applied the singular reduced
length functions of Perelman and successfully obtained some geometry structures at
infinity of a shrinker. In particular, he proved that
\begin{theorem}\label{infi}
Let $(M, g, f)$ be an $n$-dimensional shrinker with bounded curvature
operator. Then for all sequences $x_i\in M$ going to infinity along an
integral curve of $\nabla f$, there exists a subsequence also denoted
by $x_i$, such that $(M,g,x_i)$ smoothly converges to a product manifold
$\mathbb{R}\times N$, where $N$ is an $(n-1)$-dimensional shrinker.
\end{theorem}
Theorem \ref{infi} will be repeatedly used in the proof of our theorems.
In general, we do not know whether the limit manifold $\mathbb{R}\times N$
depends on the choice of sequence $x_i$. When $N$ is a quotient of the sphere,
Munteanu-Wang \cite{[MW19]} improved Naber's result. They proved that if the
cylinder $\mathbb{R}\times \mathbb{S}^{n-1}/\Gamma$ occurs as a limit for a
sequence of points going to infinity along an end of the shrinker, then the
end is smoothly asymptotic to the same cylinder.

In the end of this section, we recall the Gauss-Bonnet-Chern
formula on $4$-dimensional manifolds with boundary (e.g., \cite{Chern,Zsy}),
which is an important tool in this paper.
\begin{theorem}\label{GBC}
Let $(M, g)$ be a $4$-dimensional manifold with the boundary
$\Sigma$, then
\begin{align}\label{Gauss-Bonnet-Chern}
32\pi^2\chi(M,\Sigma)=\int_M|W|^2+2\left(\int_M (\tfrac{1}{3}R^2-|Ric|^2)
+4\int_{\Sigma}\digamma\right),
\end{align}
where
\[
\digamma=\tfrac{1}{2}RH-R_{11}H-R_{kikj}L^{ij}+\tfrac{1}{3}H^3-H|L|^2+\tfrac{2}{3}tr L^3,
\]
$L$ and $H$ are the second fundamental form and mean curvature of $\Sigma$.
Here the notations $i,j,k =2,3, 4$ denote tangential directions
and $1$ denotes the outward normal direction on $\Sigma$.
\end{theorem}

\section{Classification when the limit is $\mathbb{R}\times \mathbb{S}^3$}\label{classifica}
In this section, we mainly discuss various classifications of shrinkers
when the limit is $\mathbb{R}\times \mathbb{S}^3$.
We first consider the case of upper sectional curvature, i.e., Theorem \ref{clasfia}.
To prove Theorem \ref{clasfia}, we need the following lemma.
\begin{lemma}\label{pinc1}
Let $(M,g, f)$ be a $4$-dimensional complete noncompact shrinker with
$Ric\ge0$. If $K\le 1/4$, then
\[
\Delta_f \lambda_1\le\tfrac{1}{2}\Delta_f R+(R-\tfrac{3}{2})(R-\tfrac{5}{4})+(4-3R)\lambda_1 +3\lambda_1^2
\]
in the barrier sense.
\end{lemma}
\begin{proof}[Proof of Lemma \ref{pinc1}]
As in Section \ref{pre}, let $\{e_i\}^4_{i=1}$ be an orthonormal basis
such that they are eigenvectors of $Ric$ with respect to the corresponding
eigenvalues $\{\lambda_i\}^4_{i=1}$. We observe that
\begin{equation}
\begin{aligned}\label{eigrelat}
\lambda_1&:=R_{11}=K_{12}+K_{13}+K_{14},\\
\lambda_2&:=R_{22}=K_{21}+K_{23}+K_{24},\\
\lambda_3&:=R_{33}=K_{31}+K_{32}+K_{34},\\
\lambda_4&:=R_{44}=K_{41}+K_{42}+K_{43}.
\end{aligned}
\end{equation}
Applying these identities, \eqref{evoRia1} can be written as
\begin{align*}
\Delta_f\lambda_1&\le\lambda_1-2\lambda_2(\lambda_2-K_{23}-K_{24})
-2\lambda_3(\lambda_3-K_{23}-K_{34})-2\lambda_4(\lambda_4-K_{42}-K_{43})\\
&=\lambda_1-2\sum_{i=2}^4\lambda_i^2+2\lambda_2(K_{23}+K_{24})+2\lambda_3(K_{23}+K_{34})+2\lambda_4(K_{42}+K_{43})\\
&= \lambda_1 -2\sum_{i=2}^4\lambda_i^2 +2\sum^4_{i=2}\lambda_i(K_{23}+K_{24}+K_{34})
-(2\lambda_2K_{34}+2\lambda_3 K_{24}+2\lambda_4 K_{23})\\
&=\lambda_1-2(|Ric|^2-\lambda_1^2)+(R-\lambda_1)(R-2\lambda_1)
-(2\lambda_2 K_{34}+2\lambda_3 K_{24}+2\lambda_4 K_{23})
\end{align*}
in the barrier sense. Below, we will give a delicate upper bound
for the last term on the right hand side of the above inequality.
Indeed, using the following identity
\[
\left(\lambda_2-\tfrac 12\right)\left(K_{34}-\tfrac 14\right)
=\lambda_2 K_{34}-\tfrac 14\lambda_2-\tfrac 12K_{34}+\tfrac 18,
\]
we see that
\begin{align*}
-(2&\lambda_2K_{34}+2\lambda_3 K_{24}+2\lambda_4 K_{23})\\
&=\left[-\tfrac 12\lambda_2-K_{34}+\tfrac{1}{4}-2(\lambda_2-\tfrac{1}{2})( K_{34}-\tfrac{1}{4})\right]\\
&\quad+\left[-\tfrac{1}{2}\lambda_3 -K_{24}+\tfrac{1}{4}
-2(\lambda_3-\tfrac{1}{2})( K_{24}-\tfrac{1}{4})\right]\\
&\quad+\left[-\tfrac{1}{2}\lambda_4-K_{23}+\tfrac{1}{4}
-2(\lambda_4-\tfrac{1}{2})( K_{23}-\tfrac{1}{4})\right]\\
&=-\tfrac{1}{2}(\lambda_2+\lambda_3+\lambda_4)-(K_{23}+K_{24}+K_{34})+\tfrac{3}{4}\\
&\quad-2(\lambda_2-\tfrac{1}{2})( K_{34}-\tfrac{1}{4})-2(\lambda_3-\tfrac{1}{2})( K_{24}-\tfrac{1}{4})
-2(\lambda_4-\tfrac{1}{2})( K_{23}-\tfrac{1}{4})\\
&\le-\tfrac{1}{2}(R-\lambda_1)-\tfrac{1}{2}(R-2\lambda_1)+\tfrac{3}{4}\\
&\quad+\left[(\lambda_2-\tfrac{1}{2})^2+(\lambda_3-\tfrac{1}{2})^2+(\lambda_4-\tfrac{1}{2})^2\right]\\
&\quad+\left[(K_{23}-\tfrac{1}{4})^2+(K_{24}-\tfrac{1}{4})^2+(K_{34}-\tfrac{1}{4})^2\right],
\end{align*}
where we used the Cauchy-Schwarz inequality in the last inequality. Rearranging
the right hand side above and using the curvature assumption $K\le 1/4$, we get
that
\begin{align*}
-(2&\lambda_2K_{34}+2\lambda_3 K_{24}+2\lambda_4 K_{23})\\
&\le-R+\tfrac{3}{2}\lambda_1+\tfrac{3}{4}+(\lambda_2^2+\lambda_3^2+\lambda_4^2)
-(\lambda_2+\lambda_3+\lambda_4)+\tfrac{3}{4}\\
&\quad+(K_{23}^2+K_{24}^2+K_{34}^2)-\tfrac{1}{2}(K_{23}+K_{24}+K_{34})+\tfrac{3}{16}\\
&\le-R+\tfrac{3}{2}\lambda_1+\tfrac{15}{8}+(|Ric|^2-\lambda_1^2)-(R-\lambda_1)-\tfrac{1}{4}(R-2\lambda_1)\\
&=-\tfrac{9}{4}R+3\lambda_1+\tfrac{15}{8}+|Ric|^2-\lambda_1^2.
\end{align*}
Using this upper bound, we have that
\begin{equation}
\begin{aligned}\label{bound1}
\Delta_f \lambda_1&\le \lambda_1-2(|Ric|^2-\lambda_1^2)+(R-\lambda_1)(R-2\lambda_1)\\
&\quad-\tfrac{9}{4}R+3\lambda_1+\tfrac{15}{8}+|Ric|^2-\lambda_1^2\\
&=(\tfrac{1}{2}R-|Ric|^2)+(R^2-\tfrac{11}{4}R+\tfrac{15}{8})+(4-3R)\lambda_1 +3\lambda_1^2\\
&=\tfrac{1}{2}\Delta_f R+(R-\tfrac{3}{2})(R-\tfrac{5}{4})+(4-3R)\lambda_1 +3\lambda_1^2
\end{aligned}
\end{equation}
in the barrier sense. This completes the proof.
\end{proof}

We also need the following degenerate result for the minimal
eigenvalue of the Ricci curvature. We remark that the minimal eigenvalue
of Ricci curvature tends to zero uniformly at infinity was proved
(see Claim 4.1 in \cite{[WW]}). Here we strengthen the previous result.
\begin{lemma}\label{degen}
Let $(M,g,f)$ be a $4$-dimensional complete noncompact shrinker with $Ric\ge0$.
If $(M,g)$ is asymptotic to $\mathbb{R}\times \mathbb{S}^3$, then
\[
\lambda_1(x)=\frac{o(1)}{f(x)}.
\]
\end{lemma}
\begin{proof}[Proof of Lemma \ref{degen}]
Using $\nabla R=2Ric(\nabla f)$, we have
\begin{align*}
R-2|Ric|^2&=\Delta_fR\\
&=\Delta R-\langle\nabla f,\nabla R\rangle\\
&=\Delta R-2Ric(\nabla f,\nabla f).
\end{align*}
Since the shrinker is asymptotic to $\mathbb{R}\times \mathbb{S}^3$, we deduce that
\[
R-2|Ric|^2=o(1) \quad \text{and} \quad \Delta R=o(1).
\]
Then $Ric(\nabla f,\nabla f)=o(1)$. This implies that
\[
Ric(\tfrac{\nabla f}{|\nabla f|},\tfrac{\nabla f}{|\nabla f|})=\frac{o(1)}{|\nabla f|^2}
=\frac{o(1)}{f-R}.
\]
This together with the fact $\lambda_1\le Ric(\tfrac{\nabla f}{|\nabla f|},\tfrac{\nabla f}{|\nabla f|})$
yields the conclusion because $R(x)$ tends to a constant as $x\to\infty$.
\end{proof}

Using Lemmas \ref{pinc1} and \ref{degen}, we shall prove Theorem \ref{clasfia}
in introduction.
\begin{proof}[Proof of Theorem \ref{clasfia}]
If $(M,g,f)$ is complete noncompact and curvature operator is bounded,
by Theorem \ref{infi}, for any sequence of points $x_i\in M$ going
to infinity along an integral curve of $\nabla f$, the pointed manifolds
$(M,g,x_i)$ smoothly converge to $\mathbb{R}\times N^3$, where $N^3$
is another complete shrinker. According to the classification of
a $3$-dimensional shrinker (see \cite{[CCZ]}), $N^3$ must be $\mathbb{R}^3$,
$\mathbb{R}\times\mathbb{S}^2$ or $\mathbb{S}^3$.

When $N=\mathbb{R}^3$, we can conclude that $M=\mathbb{R}^4$. Since $Ric\ge0$,
then $R$ increases along each integral curve of $\nabla f$ outside
a compact set of $M$ (see \cite{[WW]}). Indeed, since $R+|\nabla f|^2=f$,
we may let $R_1>0$ sufficiently large such that $|\nabla f|\ne 0$ on
$M\setminus D(R_1)$. Considering an integral curve
$\gamma_{y_0}(t)$, where $t\ge 0$, to $\nabla f$ with $\gamma_{y_0}(0)=y_0$,
where $y_0\in \Sigma(R_1)$, we apply $\nabla R=2Ric(\nabla f)$
and $Ric\ge0$ to conclude that
\[
\frac{dR}{dt}(\gamma_{y_0}(t))=\langle\nabla R,\nabla f\rangle=2Ric(\nabla f,\nabla f)\ge0
\]
on $M\setminus D(R_1)$. This implies that $R$
increases along the integral curve of $\nabla f$ outside a compact set.
So $R\equiv 0$ on the whole $M$ due to $R=0$ on $\mathbb{R}\times N^3$
and the analyticity of the Ricci soliton \cite{[Kot]}. Thus $M=\mathbb{R}^4$
by \cite{[PiRS]}, which contradicts our nonflat assumption.

When $N^3=\mathbb{R}\times\mathbb{S}^2$, we see that
$\lambda_1=\lambda_2=0$, $\lambda_3=\lambda_4=1/2$ and
$K_{34}=1/2$ on $\mathbb{R}\times N^3$, which contradicts
our curvature assumption $K\le 1/4$ and hence it is impossible.

When $N^3=\mathbb{S}^3$, we see that $R=3/2$ on $\mathbb{R}\times\mathbb{S}^3$
and hence there exists a large number $a>0$ such that
\[
R-\tfrac{5}{4}>0,\quad  4-3 R<-\tfrac{1}{4}\quad
\mathrm{and} \quad \lambda_1<\tfrac{1}{100}
 \]
on $M\setminus D(a)$. Combining these
with Lemma \ref{pinc1}, we have
\begin{equation}\label{evoe}
\Delta_f \lambda_1\le\tfrac 12\Delta_f R+(R-\tfrac{3}{2})(R-\tfrac{5}{4})
\end{equation}
on $M\setminus D(a)$ in the barrier sense. Since $\lambda_1$ is a continuous and
Lipschitz function, by Proposition \ref{keyprop} in Appendix, we have
\begin{equation}\label{evoeinte}
-\int_{\Sigma(a)}\langle\nabla\lambda_1,\tfrac{\nabla f}{|\nabla f|}\rangle e^{-f}
\le \int_{M\setminus D(a)}\left[\tfrac12\Delta_f R+(R-\tfrac{3}{2})(R-\tfrac{5}{4})\right]e^{-f}.
\end{equation}

In the following, we \emph{assert} that there exists a large number $\overline{a}\ge a$ such that
\begin{equation}\label{upesinte}
\int_{\Sigma(\overline{a})}\langle\nabla\lambda_1,\tfrac{\nabla f}{|\nabla f|}\rangle e^{-f}\le 0.
\end{equation}
Indeed, we observe that, for any $b>a$, the divergence theorem gives that
\begin{equation}
\begin{aligned}\label{boun2}
\int_{D(a, b)}\langle\nabla\lambda_1,\nabla f\rangle e^{-f}
&=\int_{\Sigma(b)}\lambda_1|\nabla f|e^{-f}-\int_{\Sigma(a)}\lambda_1|\nabla f|e^{-f}
-\int_{D(a, b)}\lambda_1\Delta fe^{-f}\\
&=\int_{\Sigma(b)}\lambda_1|\nabla f|e^{-f}-\int_{\Sigma(a)}\lambda_1|\nabla f|e^{-f}
-\int_{D(a, b)}\lambda_1(2-R)e^{-f},
\end{aligned}
\end{equation}
where $D(a,b):=\{x\in M|a\le f(x)\le b\}$. From Lemma \ref{degen}, we see that
$\lambda_1=\frac{o(1)}{f}$ and hence
\[
\int_{\Sigma(b)}\lambda_1|\nabla f|e^{-f}\to 0
\]
as $b\rightarrow \infty$. Meanwhile, we clearly have
\[
-\int_{\Sigma(a)}\lambda_1|\nabla f|e^{-f}\le0.
\]
Moreover, since $2-R>0$ on $M\setminus D(a)$, then we have
\[
-\int_{D(a, b)}\lambda_1(2-R)e^{-f}\le 0.
\]
Therefore, the right hand side of \eqref{boun2} is always
non-positive as $b\rightarrow \infty$. Thus,
\[
\Lambda:=\lim_{b\rightarrow \infty}
\int_{D(a, b)}\langle \nabla \lambda_1, \nabla f\rangle e^{-f}\le 0.
\]

When $\Lambda=0$, then $\lambda_1\equiv 0$ on $D(a)$. In this case we
easily have \eqref{upesinte} by letting $\overline{a}=a$.
When $\Lambda<0$, there exists $b_1$ such that
\[
\int_{D(a,b_1)}\langle\nabla \lambda_1,\nabla f\rangle e^{-f}<0.
\]
By the co-area formula, the above inequality implies that
\begin{align*}
\int_{D(a, b_1)}\langle\nabla\lambda_1,\nabla f\rangle e^{-f}
=\int_a^{b_1}\int_{\Sigma(s)}\langle\nabla\lambda_1,\tfrac{\nabla f}{|\nabla f|}\rangle e^{-f}ds.
\end{align*}
By the mean value theorem, there exists $\overline{a}\in[a,b_1]$ such that
\eqref{upesinte} holds. Therefore we prove the assertion.

Now, combining this assertion and \eqref{evoeinte}, we obtain that

\begin{align*}
0&\le-\int_{\Sigma(\overline{a})}\langle\nabla\lambda_1,\tfrac{\nabla f}{|\nabla f|}\rangle e^{-f}\\
&\le\int_{M\setminus D(\overline{a})}\left[\tfrac{1}{2}\Delta_f R+(R-\tfrac{3}{2})(R-\tfrac{5}{4})\right]e^{-f}\\
&\le-\int_{\Sigma(\overline{a})}\tfrac{1}{2}\langle\nabla R, \tfrac{\nabla f}{|\nabla f|}\rangle e^{-f}\\
&=-\int_{\Sigma(\overline{a})} Ric(\nabla f,\tfrac{\nabla f}{|\nabla f|})e^{-f}\\
&\le 0.
\end{align*}
Equality implies that $R=3/2$ on $M\setminus D(\overline{a})$. Thus
$R=3/2$ on $M$ due to the real analyticity of the Ricci
soliton. By \cite{[FR]}, we know that $(M, g)$ is
isometric to $\mathbb{R}\times\mathbb{S}^3$.
\end{proof}

Secondly, we shall prove Theorem \ref{sumeig}. That is, for the complete
noncompact shrinker with $Ric\ge 0$ and bounded $R$, we shall apply
the curvature pinching condition $\lambda_1+\lambda_2\ge\tfrac{R}{3}$
to characterize $\mathbb{R}\times \mathbb{S}^3$.

\begin{proof}[Proof of Theorem \ref{sumeig}]
The proof is similar to the argument of Theorem \ref{clasfia}. Here we
sketch its proof. By \eqref{evoRia1} and \eqref{eigrelat}, we have
\begin{align*}
\Delta_f \lambda_1 &\leq \lambda_1 -2\lambda_2(\lambda_2-K_{23}-K_{24})-2\lambda_3(\lambda_3-K_{23}-K_{34})-2\lambda_4(\lambda_4-K_{42}-K_{43})\\
&=\lambda_1-2\sum^4_{i=2}\lambda_i^2 +2\lambda_2(K_{23}+K_{24})+2\lambda_3(K_{23}+K_{34})+2\lambda_4(K_{42}+K_{43})\\
&= \lambda_1 -2\sum^4_{i=2}\lambda_i^2 +2\sum^4_{i=2}\lambda_i(K_{23}+K_{24}+K_{34})-2(\lambda_2K_{34}+\lambda_3K_{24}+\lambda_4K_{23})\\
&\le \lambda_1-2(|Ric|^2-\lambda_1^2)+(R-\lambda_1)(R-2\lambda_1)-2\lambda_2(K_{23}+K_{24}+K_{34})\\
&=\lambda_1-2(|Ric|^2-\lambda_1^2)+(R-\lambda_1)(R-2\lambda_1)-\lambda_2(R-2\lambda_1)\\
&=R^2-2|Ric|^2-R \lambda_2+\lambda_1(1-3 R+2 \lambda_2)+4\lambda_1^2\\
&=R-2|Ric|^2+R(R-1-\lambda_2)+\lambda_1(1-3 R+2 \lambda_2)+4\lambda_1^2\\
&=\Delta_f R+R(R-1-\lambda_2-\lambda_1)+\lambda_1(1-3 R+3 \lambda_2)+4\lambda_1^2
\end{align*}
in the barrier sense, where we used $\lambda_3\ge\lambda_2$ and
$\lambda_4\ge\lambda_2$ in the above second inequality. Since
$\lambda_1+\lambda_2\ge \tfrac{R}{3}$, then
\[
\Delta_f \lambda_1\le\Delta_f R+R\left(\tfrac{2}{3}R-1\right)
+\lambda_1(1-3R+3 \lambda_2)+4\lambda_1^2
\]
in the barrier sense. Moreover, by Theorem \ref{infi}, the assumption
$\lambda_1+\lambda_2\ge \tfrac{R}{3}$ and bounded $R$ also imply that
$(M, g, f)$ must smoothly converge to $\mathbb{R}\times \mathbb{S}^3$
($\mathbb{R}^2\times \mathbb{S}^2$ is impossible).
It indicates that there exists a large number $a$ such that $\lambda_2-\tfrac{1}{2}$,
$R-\tfrac{3}{2}$ and $\lambda_1$ are sufficiently small on $M\setminus D(a)$, and
also
\[
\lambda_1(1-3R+3 \lambda_2)+4\lambda_1^2\le 0
\]
on $M\setminus D(a)$. Therefore,
\[
\Delta_f\lambda_1\le \Delta_f R+R(\tfrac{2}{3}R-1)
\]
in the barrier sense. Similar to the argument of Theorem \ref{clasfia}, there
exists $\overline{a}\ge a$ such that the integration of the above inequality
over $M\setminus D(\overline{a})$ satisfies
\begin{align*}
0&\le -\int_{\Sigma(\overline{a})}\langle\nabla\lambda_1,\tfrac{\nabla f}{|\nabla f|}\rangle e^{-f}\\
&\le\int_{M\setminus D(\overline{a})}\Delta_f R e^{-f}+\int_{M\setminus D(\overline{a})}R(\tfrac{2}{3}R-1)e^{-f}\\
&=-\int_{\Sigma(\overline{a})}\langle\nabla R, \tfrac{\nabla f}{|\nabla f|}\rangle e^{-f}\\
&=-\int_{\Sigma(\overline{a})} 2Ric(\nabla f,\tfrac{\nabla f}{|\nabla f|})e^{-f}\\
&\le 0.
\end{align*}
This implies $R=3/2$ on $M\setminus D(\overline{a})$ and hence
$(M, g, f)$ is isometric to $\mathbb{R}\times \mathbb{S}^3$
by following the same argument of Theorem \ref{clasfia}.
\end{proof}

Thirdly, we shall prove Theorem \ref{clasfi0} in introduction. That is,
we can apply the Euler number of shrinker to characterize
$\mathbb{R}\times\mathbb{S}^3$. We begin with a simple lemma,
which is well known for the 4-dimensional manifold. Use the
preceding notations and let $W_{ij}=W(e_i, e_j, e_i, e_j)$, where
$W$ denotes the Weyl curvature. Then we have
\begin{lemma}\label{eigineq}
On a $4$-dimensional manifold, we have
\begin{align}
W_{12}^2+W_{13}^2+W_{14}^2 \leq \tfrac{1}{8} |W|^2.
\end{align}
\end{lemma}
\begin{proof}[Proof of Lemma \ref{eigineq}]
Since the Weyl curvature is traceless, then
\begin{align*}
&W_{12}+W_{13}+W_{14}=0,\\
&W_{12}+W_{23}+W_{24}=0,\\
&W_{13}+W_{23}+W_{34}=0,\\
&W_{14}+W_{24}+W_{34}=0.
\end{align*}
From the above equations, we get $W_{12}=W_{34}$, $W_{13}=W_{24}$, $W_{14}=W_{23}$. Hence,
\begin{align*}
|W|^2\geq \sum_{i,j=1}^4 W_{ij}^2
 &=4\left(W_{12}^2+W_{13}^2+W_{14}^2+W_{23}^2+W_{24}^2+W_{34}^2\right)\\
&=8\left(W_{12}^2+W_{13}^2+W_{14}^2\right).
\end{align*}
This completes the proof.
\end{proof}

Now we can finish the proof of Theorem \ref{clasfi0}.
\begin{proof}[Proof of Theorem \ref{clasfi0}]
Without loss of generality, we may assume that $(M, g)$ is nonflat.
If $(M, g)$ has two ends, then $(M, g)$ splits isometrically as $\mathbb{R}\times N^3$,
where $N^3$ is a 3-dimensional closed shrinker, isometric to  $\mathbb{S}^3$.
If $(M ,g)$ has one end. Recall that
\begin{equation}\label{eigeninequ}
\Delta_f \lambda_1\leq \lambda_1-2\sum_{i=2}^4 K_{1i}\lambda_i
\end{equation}
in the barrier sense. By the curvature decomposition formula
\[
Rm=\tfrac{1}{2}Ric \odot g-\frac{R}{12}g  \odot g+W,
\]
we have
\[
-2\sum_{i=2}^4 K_{1i}\lambda_i=\tfrac{1}{3}R^2-|Ric|^2
-\tfrac{4}{3}R \lambda_1+2 \lambda_1^2-2\sum_{i=2}^4 W_{1i}\lambda_i.
\]
In the following we shall estimate the term $-2 \sum_{i=2}^4 W_{1i}\lambda_i$
in the above equality. Since Weyl curvature is traceless, we have
\begin{align*}
-2\sum_{i=2}^4 W_{1i}\lambda_i&=-2\sum_{i=2}^4 W_{1i}\left(\lambda_i-\tfrac{\lambda_2+\lambda_3+\lambda_4}{3}\right)\\
&\leq \tfrac{1}{2}\sum_{i=2}^4 \left(\lambda_i-\tfrac{\lambda_2+\lambda_3+\lambda_4}{3}\right)^2+2\sum_{i=2}^4W_{1i}^2.
\end{align*}
We observe that
\begin{align*}
\sum_{i=2}^4 \left(\lambda_i-\tfrac{\lambda_2+\lambda_3+\lambda_4}{3}\right)^2
&=\sum_{i=2}^4 \left(\lambda_i^2 -\tfrac{2}{3}(\lambda_2+\lambda_3+\lambda_4)\lambda_i
+\tfrac{1}{9}(\lambda_2+\lambda_3+\lambda_4)^2  \right)\\
&=|Ric|^2-\lambda_1^2-\tfrac{1}{3}(R-\lambda_1)^2
\end{align*}
and
\[
\sum_{i=2}^4W_{1i}^2\leq \frac{1}{8}|W|^2
\]
due to Lemma \ref{eigineq}. Hence, we have
\begin{equation}
\begin{aligned}\label{trace back}
-2\sum_{i=2}^4 K_{1i}\lambda_i\leq &\tfrac{1}{4}|W|^2+\tfrac{1}{2}\left(|Ric|^2-\lambda_1^2-\tfrac{1}{3}(R-\lambda_1)^2\right)\\
&+\tfrac{1}{3}R^2-|Ric|^2-\tfrac{4}{3}R \lambda_1+2\lambda_1^2.
\end{aligned}
\end{equation}
Now integrating \eqref{eigeninequ} over $D(a)$ and combining \eqref{trace back}
yields that
\begin{align*}
\int_{D(a)}\Delta_f \lambda_1&\leq \int_{D(a)}\lambda_1
+\int_{D(a)}\tfrac{1}{2}\left(|Ric|^2-\lambda_1^2
-\tfrac{1}{3}(R-\lambda_1)^2\right)\\
&\quad\quad+ 8\pi^2 \chi(D(a), \Sigma(a))+\tfrac{1}{2}
\int_{D(a)}(|Ric|^2-\tfrac{1}{3}R^2)-2\int_{\Sigma(a)}\digamma\\
&\quad\quad+ \int_{D(a)}\left(\tfrac{1}{3}R^2-|Ric|^2-\tfrac{4}{3}R \lambda_1+2\lambda_1^2\right)\\
&=\int_{D(a)}\left(\lambda_1-R \lambda_1+\tfrac{4}{3}\lambda_1^2\right)
-2\int_{\Sigma(a)}\digamma+8\pi^2 \chi(D(a), \Sigma(a)),
\end{align*}
where we used the Gauss-Bonnet-Chern formula \eqref{Gauss-Bonnet-Chern} in the above first
inequality.

On the other hand, since $\lambda_1$ is a continuous and Lipschitz
function, by Proposition \ref{keyprop2} in Appendix, the above inequality
indeed implies that
\begin{align*}
&\int_{\Sigma(a)}\langle\nabla \lambda_1, \tfrac{\nabla f}{|\nabla f|}\rangle-\int_{\Sigma(a)}|\nabla f| \lambda_1+\int_{D(a)}(2-R)\lambda_1\\
&\le\int_{D(a)}\left(\lambda_1-R \lambda_1+\tfrac{4}{3}\lambda_1^2\right)
-2\int_{\Sigma(a)}\digamma+8\pi^2 \chi(D(a), \Sigma(a))
\end{align*}
holds for almost everywhere sufficiently large $a$. Namely,
\begin{align*}
\int_{D(a)}\lambda_1&\leq -\int_{\Sigma(a)}\langle\nabla \lambda_1, \tfrac{\nabla f}{|\nabla f|}\rangle+\int_{\Sigma(a)}|\nabla f|\lambda_1\\
&\quad+\int_{D(a)}\tfrac{4}{3}\lambda_1^2-2\int_{\Sigma(a)}\digamma+8\pi^2 \chi(D(a), \Sigma(a)).
\end{align*}
Since the shrinker is asymptotic to $\mathbb{R}\times \mathbb{S}^3$, then
$\chi(D(a), \Sigma(a))=\chi(M)$ when $a$ is sufficiently large. We
also see that $\nabla Ric$ and the second fundamental form of $\Sigma(a)$
tends to zero as $a\rightarrow \infty$. So the boundary integral
$\int_{\Sigma(a)}\langle\nabla \lambda_1, \frac{\nabla f}{|\nabla f|}\rangle$
and  $\int_{\Sigma(a)}\digamma$ both tend to zero. Moreover, by Lemma \ref{degen},
$\int_{\Sigma(a)}|\nabla f|\lambda_1$ also tends to zero. Putting these
information together and letting $a\rightarrow \infty$, we finally obtain that
\begin{equation}\label{4chuyi3}
\int_M \lambda_1\leq \tfrac{4}{3}\int_M \lambda_1^2+8\pi^2 \chi(M).
\end{equation}
Since we assume $R\leq 3$ and $Ric\ge0$, then $\lambda_1\leq \frac{3}{4}$, which
implies $\chi(M)\geq 0$ by \eqref{4chuyi3}.

In the rest we prove the rigid part. It is clear that $\mathbb{R}\times\mathbb{S}^3$
implies $\chi(M)=0$. Now if $\chi(M)=0$, from  \eqref{4chuyi3}
and our curvature assumption, we get $\lambda_1=0$ on $M$. Combining this
with \eqref{eigeninequ} gives
\[
-2\sum_{i=2}^4 K_{1i}\lambda_i\geq 0.
\]
On the other hand, inspecting the proof from \eqref{trace back} to
\eqref{4chuyi3}, we also have
\[
\int_M -2\sum_{i=2}^4 K_{1i}\lambda_i \leq 0.
\]
Putting these together, we have
\[
\sum_{i=2}^4 K_{1i}\lambda_i\equiv 0.
\]
It is easy to see that the Ricci curvature has rank $3$ away from a compact set,
then Theorem 4.4 in \cite{[PW]} gives that $(M\setminus D, g)$ is
isometric to $\mathbb{R}\times N^3$, where $D$ is a compact set of $M$,
$N^3$ is a closed 3-dimensional shrinker (in fact is $\mathbb{S}^3/\Gamma$).
Since $\mathbb{R}\times \mathbb{S}^3/\Gamma$ has vanishing Weyl curvature and
the shrinker is analytic, the Weyl curvature vanishes on $M$.
Because $M$ is simply connected, we finally derive that $(M, g)$ is isometric
to $\mathbb{R}\times \mathbb{S}^3$, which contradicts our one end assumption.
\end{proof}

If we give a small upper bound of scalar curvature and a suitable upper bound
for the integral of the norm of Weyl curvature, we can get the following characterization
of $\mathbb{R}\times\mathbb{S}^3$. This result can be proved by using Theorem \ref{clasfi0}.
\begin{corollary}\label{Weylcon}
Let $(M, g, f)$ be a $4$-dimensional simply connected complete noncompact
Ricci shrinker with $Ric\geq 0$. If $R\leq \frac{3}{2}$ and
\begin{align*}
\int_{M}|W|^2 \leq 32\pi^2,
\end{align*}
then $(M^4, g, f)$ must be isometric to $\mathbb{R}\times \mathbb{S}^3$.
\end{corollary}

\begin{proof}[Proof of Corollary \ref{Weylcon}]
We claim that $(M, g, f)$ converges along the integral curve of
$\nabla f$ to $\mathbb{R}\times \mathbb{S}^3$. If not, then $(M, g, f)$
converges along the integral curve of $\nabla f$ to $\mathbb{R}^2\times \mathbb{S}^2$,
then the norm of Weyl curvature is bounded from below by a positive constant,
which implies $\int_{M}|W|^2=\infty$. This is impossible due to our assumption.

We observe that the Gauss-Bonnet-Chern formula with boundary \eqref{Gauss-Bonnet-Chern} gives
that
\begin{align}\label{weyl integral}
\int_{D(a)} |W|^2 =2\int_{D(a)}(|Ric|^2-\tfrac{1}{3}R^2)+32\pi^2 \chi(D(a), \Sigma(a))-8\int_{\Sigma(a)}\digamma,
\end{align}
where
\[
\digamma=\tfrac{1}{2}R H-R_{11}H-R_{kikj}L^{ij}+\tfrac{1}{3}H^3-H |L|^2 +\tfrac{2}{3}tr L^3.
\]
Since the shrinker is asymptotic to $\mathbb{R}\times \mathbb{S}^3$, then
the second fundamental form of $\Sigma(a)$ tends to zero as $a\rightarrow \infty$.
So
\[
\int_{\Sigma(a)}\digamma\rightarrow 0
\]
as $a\rightarrow\infty$. Notice that
\begin{align*}
\int_{D(a)}(|Ric|^2-\tfrac{1}{3}R^2)
&=\int_{D(a)}(|Ric|^2-\tfrac{1}{2}R+\tfrac{1}{2}R-\tfrac{1}{3}R^2)\\
&=\int_{D(a)}\left(-\tfrac{1}{2}\Delta_f R+\tfrac{1}{3}R(\tfrac{3}{2}-R)\right)\\
&\geq \int_{D(a)}\left(-\tfrac{1}{2}\Delta R+\tfrac{1}{2}\langle\nabla f, \nabla R\rangle+\tfrac{1}{3}R(\tfrac{3}{2}-R)\right)\\
&\geq \int_{D(a)}\left(-\tfrac{1}{2}\Delta R+Ric(\nabla f, \nabla f)+\tfrac{1}{3}R(\tfrac{3}{2}-R)\right)\\
&\geq -\tfrac{1}{2}\int_{\Sigma(a)}\langle\nabla R, \tfrac{\nabla f}{|\nabla f|}\rangle+\int_{D(a)}\tfrac{1}{3}R(\tfrac{3}{2}-R).
\end{align*}
Also notice that
\[
\int_{\Sigma(a)}\langle\nabla R,\tfrac{\nabla f}{|\nabla f|}\rangle\to 0
\]
as $a\rightarrow\infty$. Using these information in \eqref{weyl integral},
we obtain that
\begin{align}\label{oulashuguanxishi}
\int_{M}|W|^2\geq \tfrac{2}{3}\int_{M}R(\tfrac{3}{2}-R)+32\pi^2 \chi(M).
\end{align}
Since $\int_{M}|W|^2\leq 32\pi^2$ and $R\leq \frac{3}{2}$, then $\chi(M)\leq 1$.
If $\chi(M)=1$, then (\ref{oulashuguanxishi}) implies $R=\frac{3}{2}$, which
implies $(M, g, f)$ is isometric to $\mathbb{R}\times \mathbb{S}^3$ by
\cite{[FR]}. However the Euler characteristic of $\mathbb{R}\times \mathbb{S}^3$
is zero. So this situation does not occur. Therefore $\chi(M)\leq 0$. Combining this
with Theorem \ref{clasfi0} implies that $(M, g, f)$ is isometric to
$\mathbb{R}\times \mathbb{S}^3$.
\end{proof}

In the end of this section, we point out that the argument of Theorem \ref{clasfi0}
by the Gauss-Bonnet-Chern formula gives a by-product about
volume comparison result for compact shrinkers.

\begin{theorem}\label{shrcomp}
Let $(M, g, f)$ be a $4$-dimensional simply connected
compact shrinker. If $\frac{1}{4}<\lambda_1(x)\le\frac{1}{2}$ for any $x\in M$,
then
\begin{align}\label{volume comparison}
Vol(M)\leq 48\pi^2 \chi(M).
\end{align}
Moreover, equality holds if and only if $(M, g)$ is isometric to $\mathbb{S}^4$.
\end{theorem}
\begin{proof}[Proof of Theorem \ref{shrcomp}]
The proof is exactly the same as the argument of Theorem \ref{clasfi0}.
Here the compact case is much easier because there is no boundary term.
In fact, we follow the argument of Theorem \ref{clasfi0} and still have
the formula \eqref{4chuyi3}, that is,
\[
\int_M \lambda_1\leq \tfrac{4}{3}\int_M \lambda_1^2+ 8\pi^2 \chi(M).
\]
Since
\[
\lambda_1-\tfrac{4}{3} \lambda_1^2\ge \tfrac{1}{6}
\]
for $\tfrac{1}{4}<\lambda_1\le\tfrac{1}{2}$, we immediately get
\begin{align*}
Vol(M)\leq 48\pi^2 \chi(M),
\end{align*}
which proves the inequality. Next, we consider the equality case.
If $Vol(M)= 48\pi^2 \chi(M)$, then
\[
\lambda_1=\tfrac{1}{2}
\quad \text{and}\quad
R\geq 4 \lambda_1=2.
\]
Since $\Delta f+R=2$, we get $\Delta f\leq 0$, which implies $f$ is
constant, i.e., $Ric=\frac{1}{2}g$.  Combining this with the
Gauss-Bonnet-Chern formula on closed $4$-dimensional manifolds, we have
\[
\int_M |W|^2 =2 \int_M (|Ric|^2-\frac{1}{3}R^2)+32\pi^2\chi(M),
\]
which implies
\[
\int_M |W|^2 =-\frac{2}{3}Vol(M)+32\pi^2\chi(M).
\]
This implies $W\equiv 0$ and hence $(M, g)$ is isometric
to $\mathbb{S}^4$.
\end{proof}

\begin{remark}
The volume comparison \eqref{volume comparison} can be easily obtained on Einstein manifolds
with $Ric=\frac{1}{2}g$ due to the Gauss-Bonnet-Chern formula. In our setting,
when $\frac{1}{4}g<Ric\le \frac{1}{2}g$, we clearly see that volume comparison
\eqref{volume comparison} is better than the classical Bishop-Gromov volume
comparison.
\end{remark}

\section{Classification when the asymptotic limit is $\mathbb{R}^2\times \mathbb{S}^2$}
\label{R2S2}

In this section, we shall study various classifications of shrinkers
when the limit is $\mathbb{R}^2\times \mathbb{S}^2$. We first consider
the case of sharp upper sectional curvature, i.e., Theorem \ref{clasfi}.
To prove this theorem, we need the following result.
\begin{lemma}\label{pin2}
Let $(M,g, f)$ be a $4$-dimensional complete noncompact shrinker with
$Ric\ge0$. If $K\le 1/2$, then
\begin{equation}\label{inbeint}
\Delta_f (\lambda_1+\lambda_2)\le\Delta_f R+(\lambda_1+\lambda_2)(R-1)-4\lambda_1\lambda_2
\end{equation}
in the barrier sense.
\end{lemma}
\begin{proof}[Proof of Lemma \ref{pin2}]
Recall that
\begin{align*}
\Delta_f (\lambda_1+\lambda_2)&\leq \lambda_1-2(K_{12}\lambda_2+K_{13}\lambda_3+K_{14}\lambda_4)
+\lambda_2-2(K_{21}\lambda_1+K_{23}\lambda_3+K_{24}\lambda_4)\\
&=(\lambda_1+\lambda_2)-2\left[K_{12}(\lambda_1+\lambda_2)
+\lambda_3(K_{13}+K_{23})+\lambda_4(K_{14}+K_{24})\right]\\
&=(\lambda_1+\lambda_2)-2K_{12}(\lambda_1+\lambda_2)
+2K_{34}(\lambda_3+\lambda_4)-2(\lambda^2_3+\lambda^2_4)
\end{align*}
in the barrier sense, where we used $K_{13}+K_{23}=\lambda_3-K_{34}$ and
$K_{14}+K_{24}=\lambda_4-K_{34}$ in the last equality. We then apply
\[
2K_{34}=2K_{12}+(\lambda_3+\lambda_4-\lambda_1-\lambda_2)
\]
to further get that
\begin{equation}
\begin{aligned}\label{eigenevo}
\Delta_f (\lambda_1+\lambda_2)&\le(\lambda_1+\lambda_2)-(\lambda_1+\lambda_2)(\lambda_3+\lambda_4)
+2K_{12}[(\lambda_3+\lambda_4)-(\lambda_1+\lambda_2)]\\
&\quad+\left[R-(\lambda_1+\lambda_2)\right]^2-2\left[|Ric|^2-(\lambda^2_1+\lambda^2_2)\right]
\end{aligned}
\end{equation}
in the barrier sense. Since $K\le 1/2$, then
\[
2K_{34}=2K_{12}+(R-2\lambda_1-2\lambda_2)\le 1,
\]
which implies
\[
2K_{12}\le 1-R+2(\lambda_1+\lambda_2).
\]
Substituting this estimate into \eqref{eigenevo} and using $R-2|Ric|^2=\Delta_fR$, we
finally get the desired result.
\end{proof}

We are now in position to prove Theorem \ref{clasfi} by Lemma \ref{pin2}.
\begin{proof}[Proof of Theorem \ref{clasfi}]
By the assumption of scalar curvature, we get that $(M, g)$ converges along
the integral curve of $\nabla f$ to $\mathbb{R}\times (\mathbb{R}\times \mathbb{S}^2)/\Gamma$
for some $\Gamma$.

\vspace{0.5em}

\textbf{Claim 1}: $(\lambda_1+\lambda_2)\to 0$ and $R\to 1$ uniformly at infinity.

If Claim 1 is false, then there is a sequence of points
$q_i\to\infty$ but $(\lambda_1+\lambda_2)(q_i)\geq \delta$ for some $\delta>0$.
By Theorem 1.4 in \cite{[Na]}, for the $4$-dimensional shrinker $(M,g,f)$ with
bounded curvature (bounded scalar curvature assures the bounded curvature operator;
see \cite{[MWa]}), the associated Ricci flow $(M, g(t), q_i)$ defined on $(-\infty, 1)$, where
\[
g(t):=(1-t)\phi^*_tg,\quad \frac{d\phi_t}{dt}=\frac{\nabla f}{1-t}
\quad  \text{and}\quad \phi_0=\mathrm{Id},
\]
is $\kappa$-noncollapsed, where $\kappa=\kappa(n, V_f(M))$. By Hamilton's
Cheeger-Gromov compactness theorem \cite{[Ham]}, the associated Ricci flow
sub-converges to an ancient $\kappa$-noncollapsed solution
$(M_{\infty}, g_{\infty}(t), q_{\infty})$ . Since the curvature is
bounded, by \eqref{equat} and Theorem \ref{pote}, $|\nabla f(x)|\to \infty$
at the linear growth rate as $x\to \infty$. Consider a sequence of functions
\[
f_i(x):=\frac{f(x)-f(q_i)}{|\nabla f(q_i)|}.
\]
It satisfies
\[
|\nabla f_i(q_i)|=1 \quad\text{and} \quad
|\mathrm{Hess}\,f_i|=\frac{|\tfrac{1}{2}g-Ric(g)|}{|\nabla f(q_i)|}.
\]
on $(M^n,g(t))$. Since $|\nabla f(x)|\to \infty$ as $x\to\infty$,
$|\nabla f(q_i)|\to\infty$ as $q_i\to\infty$. Combining this with the bounded
curvature condition, we get $|\mathrm{Hess}\, f_i|\to 0$ uniformly as
$i\to\infty$. Hence along the convergence of $(M^n, g(t), q_i)$, the sequence
of functions $f_i(x)$ smoothly converges to a limit $f_\infty$ satisfying
$|\nabla f_\infty(q_{\infty})|=1$ and $\mathrm{Hess}\, f_\infty\equiv 0$
on $(M_{\infty}, g_{\infty}(t))$. This implies that
$(M_{\infty}, g_{\infty}(t), q_{\infty})$ isometrically splits
as $(\mathbb{R}\times N^3, g_\mathbb{R}+g_{N^3(t)})$, where $(N^3, g_{N^3(t)})$
is an $3$-dimensional ancient and $\kappa$-noncollapsed solution.
By the recent work of  \cite{Angenent-Brendle-Daskalopoulos-Sesum,Brendle,Brendle-Daskalopoulos-Sesum},
the compact case also by \cite{Bamler-Kleiner}, there are four possibilities for $N^3$.

(1). If $N^3$ is isometric to $S^3/\Gamma$, by the main result in Munteanu-Wang \cite{[MW19]},
$(M, g, f)$ converges uniformly to $\mathbb{R}\times \mathbb{S}^3/\Gamma$. This is
contradiction with our scalar curvature condition!

(2). If $N^3$ is isometric to a finite quotient of the Type II ancient solution
constructed by Perelman, this indicates that there is a sequence level sets
$\Sigma(s_i)$ whose volumes are uniformly bounded above. Since $(M, g, f)$ converges to
$\mathbb{R}\times (\mathbb{R}\times \mathbb{S}^2)/\Gamma$ along the integral
curve of $\nabla f$, the volume of the level set $\Sigma(s)$ tends to infinity. So
it is impossible.

(3). If $N^3$ is isometric to the Bryant soliton, since $Ric\geq 0$ and
$\langle\nabla R, \nabla f\rangle=2Ric(\nabla f, \nabla f)$, then $R$ is
increasing along the integral curve of $\nabla f$. So $R$ is bounded from
below by some positive constant. This is impossible due to a fact that
the scalar curvature of Bryant soliton does not have a positive lower
bound.

(4). If $N^3$ is isometric to $(\mathbb{R}\times \mathbb{S}^2)/\Gamma$,
 then $(\lambda_1+\lambda_2)(q_\infty)=0$.

To sum up, $(\lambda_1+\lambda_2)\to 0$ uniformly at infinity.
We similarly argue to derive that $R\to 1$ uniformly at infinity.
So Claim 1 follows.

\vspace{0.5em}

\textbf{Claim 2}: There exists $b$ large enough such that
\[
\int_{\Sigma(b)}\langle\nabla (\lambda_1+\lambda_2), \tfrac{\nabla f}{|\nabla f|}\rangle \le 0.
\]
To prove this claim, we first consider the following one parameter of diffeomorphism,
\[
\frac{\partial F}{\partial t}=\frac{\nabla f}{|\nabla f|^2}
\quad \text{with}\quad F(x, a_0)=x \in \Sigma(a_0).
\]
The advantage of map $F$ is that it maps the level set of $f$ to another level set,
in particular $f(F(x, t))=t$ for any $x\in\Sigma(a_0)$. Let $\{x_1, x_2, x_3\}$
be a local coordinate chart of $\Sigma(a_0)$. On $\Sigma(t)$, let
\[
g(\tfrac{\partial }{\partial x_i},\tfrac{\partial}{\partial x_j}):=g(\tfrac{\partial F}{\partial x_i}, \tfrac{\partial F}{\partial x_j})
\quad \text{and}\quad
dvol_{\Sigma(t)}=\sqrt{det(g_{ij})}dx,
\]
where $dx=dx_1\wedge dx_2\wedge dx_3$. For simplicity, denote by $E=\frac{\nabla f}{|\nabla f|}$.
We compute the derivatives of $dvol_{\Sigma(t)}$,
\begin{align*}
\frac{\partial}{\partial t}dvol_{\Sigma(t)}&=\frac{\partial}{\partial t}\sqrt{det(g_{ij})}dx\\
&= g^{ij}\langle \nabla_{\frac{\partial F}{\partial x_i}}\tfrac{\partial F}{\partial t}, \tfrac{\partial F}{\partial x_j}\rangle dvol_{\Sigma(t)}\\
&= g^{ij}\langle \nabla_{\frac{\partial F}{\partial x_i}}\tfrac{\nabla f}{|\nabla f|^2}, \tfrac{\partial F}{\partial x_j}\rangle dvol_{\Sigma(t)}\\
&=\frac{1}{|\nabla f|^2}g^{ij}\mathrm{Hess}\, f(\tfrac{\partial F}{\partial x_i}, \tfrac{\partial F}{\partial x_j})dvol_{\Sigma(t)}\\
&=\frac{1}{|\nabla f|^2} g^{ij}\left(\tfrac{1}{2}g(\tfrac{\partial F}{\partial x_i}, \tfrac{\partial F}{\partial x_j})-Ric(\tfrac{\partial F}{\partial x_i}, \tfrac{\partial F}{\partial x_j})\right)dvol_{\Sigma(t)}\\
&=\frac{1}{|\nabla f|^2} \left(\tfrac{3}{2}-R+Ric(E, E)\right)dvol_{\Sigma(t)}.
\end{align*}
We compute the derivative of $|\nabla f|$,
\begin{align*}
\frac{\partial }{\partial t}|\nabla f|
&=\langle \nabla f, \nabla f\rangle ^{-\frac{1}{2}}\langle\nabla_{\tfrac{\partial F}{\partial t}}\nabla f, \nabla f\rangle\\
&=\frac{1}{|\nabla f|}\left(\tfrac{1}{2}-Ric(E, E)\right).
\end{align*}
Combining the above two parts, we have
\begin{align*}
\frac{\partial}{\partial t}\left(\tfrac{|\nabla f|}{t}dvol_{\Sigma(t)}\right)
&=\tfrac{|\nabla f|}{t}\cdot\tfrac{1}{|\nabla f|^2} \left(\tfrac{3}{2}-R+Ric(E, E)\right)dvol_{\Sigma(t)}\\
&\quad +\left(\tfrac{-|\nabla f|}{t^2}+\tfrac{1}{t}\cdot\tfrac{1}{|\nabla f|}(\tfrac{1}{2}-Ric(E, E))\right)dvol_{\Sigma(t)}\\
&=\frac{1}{t|\nabla f|} \left(\tfrac{3}{2}-R+Ric(E, E)-\tfrac{|\nabla f|^2}{t}+\tfrac{1}{2}-Ric(E, E)\right)dvol_{\Sigma(t)}\\
&=\frac{1}{t|\nabla f|}\left(2-R-\tfrac{|\nabla f|^2}{t}\right)dvol_{\Sigma(t)}.
\end{align*}
Since $Ric\ge0$, $R$ is increasing along the integral curve of $\nabla f$.
Moreover, the asymptotic limit of the shrinker at infinity is
$\mathbb{R}\times(\mathbb{R}^2\times \mathbb{S}^2)/\Gamma$ for some $\Gamma$
whose $R=1$ by \cite{[Na]}. Hence $R\leq 1$ on $M\setminus D(a)$. Combining this
with $R+|\nabla f|^2=f$, we get $\frac{|\nabla f|^2}{t}\leq 1$. So
\begin{align}\label{volume derivative inequality}
\frac{\partial}{\partial t}\left(\tfrac{|\nabla f|}{t}dvol_{\Sigma(t)}\right) \geq 0.
\end{align}

On the other hand, since the curvature is bounded, the shrinker is $\kappa$-noncollapsed \cite{li-wang}.
So we obtain the volume of any ball of radius one has a positive uniform lower bound, thus the injectivity radius of $(M^4, g)$ is bounded from below by some positive constant. Then we can apply
Zhu's volume estimate (Corollary 1.9 in \cite{[Zhubo]}) to derive that
\[
Vol(B(p, r))\leq C\cdot r
\]
for some $C>0$. Combining this with the estimate: $f(x)\sim \frac{1}{4}d(x, p)^2$
in Theorem \ref{pote}, we have that
\[
Vol(D(t))\leq C_1\cdot t
\]
for some $C_1>0$. So there is a sequence $t_i\rightarrow \infty$ such that
\[
\frac{d}{dt}Vol(D(t))|_{t=t_i}\leq C_2
\]
for some $C_2>0$. Using this, and combining the coarea formula
$\frac{d}{dt}Vol(D(t))=\int_{\Sigma(t)}\frac{1}{|\nabla f|}$
and $|\nabla f|\sim \sqrt{f}$ (this is due to $R+|\nabla f|^2=f$ and $R$ is bounded), we have
\[
Vol(\Sigma(t_i))\leq C_3 \sqrt{t_i}
\]
for some $C_3>0$. By Claim 1, $\lambda_1+\lambda_2$ tends to zero at infinity.
Hence the quantity
\begin{align*}
\int_{\Sigma(t)}(\lambda_1+\lambda_2)\tfrac{|\nabla f|}{t}dvol_{\Sigma(t)}
\end{align*}
tends to zero along the sequence $t_i$. Thus there exists $b$ large enough such that
\begin{align}\label{volume derivative}
\frac{d}{dt}\int_{\Sigma(t)}(\lambda_1+\lambda_2)\tfrac{|\nabla f|}{t}dvol_{\Sigma(t)}\Big|_{t=b}\leq 0.
\end{align}

Now, using \eqref{volume derivative inequality} we compute that
\begin{align*}
\frac{d}{dt}\int_{\Sigma(t)}(\lambda_1+\lambda_2)\tfrac{|\nabla f|}{t}dvol_{\Sigma(t)}\Big|_{t=b}
&=\int_{\Sigma(b)} \langle\nabla(\lambda_1+\lambda_2), \tfrac{\nabla f}{|\nabla f|^2}   \rangle\tfrac{|\nabla f|}{b}dvol_{\Sigma(b)}\\
&\quad+\int_{\Sigma(b)}(\lambda_1+\lambda_2)\frac{\partial}{\partial t}\left(\tfrac{|\nabla f|}{t}dvol_{\Sigma(t)}\right)\\
&\geq\int_{\Sigma(b)}\langle\nabla(\lambda_1+\lambda_2),\tfrac{\nabla f}{|\nabla f|}   \rangle\tfrac{1}{b}dvol_{\Sigma(b)}.
\end{align*}
This result together with \eqref{volume derivative} yields
\[
\int_{\Sigma(b)}\langle\nabla(\lambda_1+\lambda_2),\tfrac{\nabla f}{|\nabla f|}\rangle
dvol_{\Sigma(b)}\leq 0.
\]
This finishes the proof of Claim 2.

\vspace{0.5em}

%\begin{align*}
%&\frac{d}{dt}\int_{\Sigma(t)} u|\nabla f|\\
%&=\frac{d}{dt}\int_{\Sigma(a_0)} u|\nabla f|\frac{A(t, x)}{A(a_0, x)}A(a_0, x)\\
%&=\int_{\Sigma(a_0)}\langle\nabla u, \frac{\nabla f}{|\nabla f|^2}\rangle|\nabla f|\frac{A(t, x)}{A(a_0, x)}A(a_0, x)\\
%&\quad+\int_{\Sigma(a_0)}\frac{1}{|\nabla f|}\left(\frac{1}{2}-Ric(\frac{\nabla f}{|\nabla f|}, \frac{\nabla f}{|\nabla f|})\right)\frac{A(t, x)}{A(a_0, x)}A(a_0, x)\\
%&\quad  \int_{\Sigma(a_0)}      \frac{1}{|\nabla f|^2} \left(\frac{3}{2}-S+Ric(\frac{\nabla f}{|\nabla f|}, \frac{\nabla f}{|\nabla f|}\right)A(t, x)
%\end{align*}

Next, we will apply Lemma \ref{pin2} and Claim 2 to continue our proof.
Since $Ric\ge 0$ and $R\leq 1$ on  $M\setminus D(b)$. Integrating \eqref{inbeint} over
$M\setminus D(b)$ gives that
\begin{equation}
\begin{aligned}\label{intebp}
\int_{M\setminus D(b)}\Delta_f (\lambda_1+\lambda_2)e^{-f}
&\le\int_{M\setminus D(b)}\Delta_f R e^{-f}+\int_{M\setminus D(b)}\left[(\lambda_1+\lambda_2)(R-1)-4\lambda_1\lambda_2\right]e^{-f}\\
&\le-\int_{\Sigma(b)}\langle\nabla R, \tfrac{\nabla f}{|\nabla f|}\rangle e^{-f}\\
&=-2\int_{\Sigma(b)}Ric(\nabla f, \tfrac{\nabla f}{|\nabla f|}) e^{-f}\\
&\le 0.
\end{aligned}
\end{equation}
Since $\lambda_1+\lambda_2$ is continuous and Lipschitz, from Proposition
\ref{keyprop} in Appendix, we see that the left hand side term in
\eqref{intebp} can be regarded as
\[
\int_{M\setminus D(b)}\Delta_f (\lambda_1+\lambda_2)e^{-f}
=-\int_{\Sigma(b)}\langle\nabla (\lambda_1+\lambda_2), \tfrac{\nabla f}{|\nabla f|}\rangle e^{-f}\ge0
\]
for almost everywhere $b$, where we used Claim 2 in the second inequality.
Hence all inequalities in \eqref{intebp} become equalities and we conclude
\[
\int_{M\setminus D(b)}(\lambda_1+\lambda_2)(R-1)e^{-f}
=\int_{M\setminus D(b)}\lambda_1\lambda_2e^{-f}=\int_{M\setminus D(b)}\Delta_f R e^{-f}=0.
\]
This implies
\begin{equation}\label{ident}
\lambda_1 \lambda_2\equiv0\quad \text{and}\quad (\lambda_1+\lambda_2)(R-1)\equiv 0
\end{equation}
on $M\setminus D(b)$. Now we shall apply \eqref{ident} to prove our
theorem by the the following two cases.

Case 1.  If there exists $q\in M\backslash D(b)$ such that $\lambda_2(q)>0$,
then $\lambda_2(x)>0$ on $B(q, \delta)$ for some $\delta>0$. By \eqref{ident},
we get $R\equiv 1$ on $B(q, \delta)$ and hence $R\equiv 1$ on $M$ by the
analyticity of the soliton. By \cite{Cheng-Zhou}, $(M, g)$ is isometric to
$\mathbb{R}^2\times \mathbb{S}^2$. This contradicts with $\lambda_2(q)>0$.

Case 2. If $\lambda_1+\lambda_2\equiv 0$ on $M\backslash D(b)$, then from
Lemma \ref{pin2} we get $\Delta_f R\equiv 0$ on $M\setminus D(b)$. Therefore,
\begin{align*}
\int_{M\setminus D(b)}|\nabla R|^2 e^{-f}
&=\int_{M\setminus D(b)} div(R\nabla R e^{-f})-\Delta_f R\cdot S e^{-f}\\
&=-\int_{\Sigma(b)}\langle R \nabla R, \tfrac{\nabla f}{|\nabla f|}\rangle e^{-f}\\
&=-\int_{\Sigma(b)} 2 R\cdot Ric(\nabla f, \tfrac{\nabla f}{|\nabla f|}) e^{-f}\\
&\leq 0.
\end{align*}
So $\nabla R\equiv 0$ and $R$ is constant on $M\setminus D(b)$.
Hence $R\equiv 1$ on $M$ due to the asymptotic limit is $\mathbb{R}^2\times \mathbb{S}^2$
and the analyticity of the soliton. We again get that $(M, g)$ is isometric to
$\mathbb{R}^2\times \mathbb{S}^2$ by \cite{Cheng-Zhou}.
\end{proof}

In the rest of this section, we shall prove Theorem \ref{bi-Ricci} by adopting
a similar argument of Theorem \ref{clasfi}. The main difference is that
we shall use the formula $\Delta_f (\lambda_1+\lambda_2)$ instead of
$\Delta_f\lambda_1$.

\begin{proof}[Proof of Theorem \ref{bi-Ricci}]
Since the proof essentially follows the argument of Theorem \ref{clasfi}, we
only sketch some main proof steps. Let $\{e_i\}^4_{i=1}$ be an orthonormal
basis of $(M,g)$ such that they are the eigenvectors of $Ric$ with respect
to the corresponding eigenvalues $\{\lambda_i\}^4_{i=1}$. Since the bi-Ricci
curvature is nonnegative, then
\[
Ric(e_1, e_1)+Ric(e_2, e_2)-K_{12}\ge 0.
\]
Substituting this into \eqref{eigenevo}, we see that
\begin{equation}
\begin{aligned}\label{inbeint2}
\Delta_f (\lambda_1+\lambda_2)
&\leq(\lambda_1+\lambda_2)-(\lambda_1+\lambda_2)(\lambda_3+\lambda_4)
+2(\lambda_1+\lambda_2)(\lambda_3+\lambda_4-\lambda_1-\lambda_2)\\
 &\quad+(R-\lambda_1-\lambda_2)^2-2(|Ric|^2-\lambda_1^2-\lambda_2^2)\\
 &=(\lambda_1+\lambda_2)+R^2-R(\lambda_1+\lambda_2)-2(\lambda_1+\lambda_2)^2-2|Ric|^2+2(\lambda_1^2+\lambda_2^2)\\
 &= \Delta_fR+(1-R)(\lambda_1+\lambda_2)+R(R-1)-4 \lambda_1 \lambda_2\\
 &=\Delta_fR+(R-1)(\lambda_3+\lambda_3)-4 \lambda_1 \lambda_2
\end{aligned}
\end{equation}
in the barrier sense. Under our assumption of scalar curvature, $(M, g)$
still converges along the integral curve of $\nabla f$ to
$\mathbb{R}\times(\mathbb{R}\times \mathbb{S}^2)/\Gamma$. Meanwhile, similar
to the proof of Theorem \ref{clasfi}, we also have Claim 1 and Claim 2 in
the proof process of Theorem \ref{clasfi}. Then, integrating \eqref{inbeint2} over
$M\setminus D(b)$ gives that
\begin{equation}
\begin{aligned}\label{intebp2}
\int_{M\setminus D(b)}\Delta_f (\lambda_1+\lambda_2)e^{-f}
&\le\int_{M\setminus D(b)}\Delta_f R e^{-f}+\int_{M\setminus D(b)}\left[(R-1)(\lambda_3+\lambda_4)-4 \lambda_1 \lambda_2\right]e^{-f}\\
&\le-\int_{\Sigma(b)}\langle\nabla R, \tfrac{\nabla f}{|\nabla f|}\rangle e^{-f}\\
&=-2\int_{\Sigma(b)}Ric(\nabla f, \tfrac{\nabla f}{|\nabla f|}) e^{-f}\\
&\le 0,
\end{aligned}
\end{equation}
where we used $R\le 1$ on $M\setminus D(b)$ in the above second inequality.
Similar to the argument of Theorem \ref{clasfi}, we can conclude that
\[
\int_{M\setminus D(b)}\Delta_f (\lambda_1+\lambda_2)e^{-f}
=-\int_{\Sigma(b)}\langle\nabla (\lambda_1+\lambda_2), \tfrac{\nabla f}{|\nabla f|}\rangle e^{-f}\ge0
\]
holds for almost everywhere $b$. Therefore all inequalities in \eqref{intebp2}
become equalities. Thus,
\begin{equation}\label{ident2}
\lambda_1 \lambda_2\equiv0\quad \text{and}\quad (\lambda_3+\lambda_4)(R-1)\equiv 0.
\end{equation}
on $M\setminus D(b)$. Since $\lambda_3+\lambda_4\neq 0$ on $M\setminus D(b)$,
this forces $R\equiv 1$ on $M\setminus D(b)$. Then conclusion follows by
the same argument of Theorem \ref{clasfi}.
\end{proof}

\section{Classification when $R=1$}\label{csc}
In \cite{Cheng-Zhou} Cheng-Zhou proved the following result.
\begin{theorem}[\cite{Cheng-Zhou}]\label{ChengZh}
Let $(M,g, f)$ be a $4$-dimensional simply connected complete noncompact
shrinker with $R=1$. Then $(M, g, f)$ is isometric to $\mathbb{R}^2\times \mathbb{S}^2$.
\end{theorem}
Recently, inspired by Petersen-Wylie's work \cite{[PW2]}, Ou-Qu-Wu \cite{Ou-Qu-Wu}
gave a new proof of Cheng-Zhou's result by applying maximum principle to
$\Delta_f(\lambda_1+\lambda_2)$. In this section, we shall provide an alternative
proof based on the integration method repeatedly used in the above sections.
Compared to the proof of Theorem \ref{clasfi}, in the setting of $R=1$, we luckily
effectively control the intractable curvature term $K_{12}$.

%At first, we want to explain the known results and methods about Cao's conjecture.
%\begin{theorem}[\cite{[PW2]}] A Ricci shrinker is rigid iff it has constant scalar curvature and is radially flat, that is, the sectional curvature
%\begin{align*}
% K(\nabla f, \cdot)=0.
 %\end{align*}
 %\end{theorem}
%Later, Fern\'{a}ndez-L\'{o}pez and Garc\'\i a-R\'\i o \cite{[FR]} characterize the rigidity using the rank of Ricci curvature.
%\begin{theorem}[\cite{[FR]}]\label{constant rank} A  Ricci shrinker is rigid iff it has constant scalar curvature and the Ricci curvature has constant rank.
%\end{theorem}
%Recently, Cheng-Zhou \cite{Cheng-Zhou} proved a four dimensional Ricci shrinker with $S=1$ is rigid. They applied $\Delta_f$ to the  quantity
%\begin{align*}
%tr(Ric^3)-\frac{1}{4},
%\end{align*}
%and they got the following nice inequality
%\begin{align*}
%\Delta_f  \left(f( tr(Ric^3)-\frac{1}{4})\right)\geq 9 f( tr(Ric^3)-\frac{1}{4}),
%\end{align*}
%at last they used the integration by parts to derive that
%\begin{align*}
%tr(Ric^3)-\frac{1}{4}=0
%\end{align*}
%over $M$, and this implies that $\lambda_1+\lambda_2=0$, so the Ricci curvature has rank $2$, finally they obtain the rigidity by Theorem \ref{constant rank}.

%We want to point remark that  $(3.11)$ in Cheng-Zhou \cite{Cheng-Zhou} gives that
%\begin{align*}
%\frac{1}{3}tr(Ric^3)-\frac{1}{12}=\lambda_2\lambda_3\lambda_4,
%\end{align*}
%this implies that the quantity they used is $\sigma_3(Ric)$, since $\lambda_1=0$ in this situation.

%\subsection{Rigidity when scalar curvature $R=1$}

In order to reprove Theorem \ref{ChengZh} by the integration method, we need some lemmas.
We first need a sharp estimate for the Weyl curvature by a simple algebraic inequality.
\begin{lemma}\label{1/12lemma}
For the Weyl curvature tensor in a $4$-dimensional manifold, we have
\[
W_{12}^2\leq \frac{1}{12}|W|^2.
\]
\end{lemma}
\begin{proof}[Proof of Lemma \ref{1/12lemma}]
Since $W_{12}+W_{13}+W_{14}=0$, by the Cauchy-Schwarz inequality, we have
\begin{align*}
W^2_{12}+W^2_{13}+W^2_{14}&\ge W^2_{12}+\frac 12(W_{13}+W_{14})^2\\
&=W^2_{12}+\frac 12W^2_{12}=\frac 32W^2_{12}.
\end{align*}
This implies
\[
W_{12}^2\leq \frac{2}{3}(W_{12}^2+W_{13}^2+W_{14}^2)
\leq \frac{2}{3}\times \frac{1}{8}|W|^2=\frac{1}{12}|W|^2,
\]
where we used Lemma \ref{eigineq} in the second inequality.
\end{proof}

From the preceding proof, we know that the integration argument requires the
nonnegativity of Ricci curvature. This nonnegativity can be achieved by the
scalar curvature assumption.
\begin{lemma}[\cite{[FR]}]
Let $(M, g, f)$ be a $4$-dimensional complete noncompact shrinker with $R=1$. Then $Ric\geq 0$.
\end{lemma}

We are now in position to prove Theorem \ref{ChengZh}.

\begin{proof}[Proof of Theorem \ref{ChengZh}]
Adopting the preceding notations and formulas, recall that
\begin{equation}
\begin{aligned}\label{kaishide}
\Delta_f (\lambda_1+\lambda_2)&\leq (\lambda_1+\lambda_2)-(\lambda_1+\lambda_2)(\lambda_3+\lambda_4)\\
&\quad+2K_{12}(\lambda_3+\lambda_4-\lambda_1-\lambda_2)+(\lambda_3+\lambda_4)^2
-2(\lambda_3^2+\lambda_4^2)
\end{aligned}
\end{equation}
in the barrier sense. By the curvature decomposition formula, we have
\[
K_{12}=\frac{1}{2}(\lambda_1+\lambda_2)-\frac{1}{6}+W_{12}.
\]
Applying it and the identity $2|Ric|^2=R=1$ gotten by \eqref{Sequat},
we compute that
\begin{align*}
\Delta_f (\lambda_1+\lambda_2)&\leq
(\lambda_1+\lambda_2)-(\lambda_1+\lambda_2)(1-\lambda_1-\lambda_2)+(\lambda_1+\lambda_2-\tfrac{1}{3}+2W_{12})\\
&\quad-4K_{12}(\lambda_1+\lambda_2)+(1-\lambda_1-\lambda_2)^2-2(\tfrac{1}{2}-\lambda_1^2-\lambda_2^2)\\
&=-(\lambda_1+\lambda_2)-\tfrac{1}{3}+2W_{12}-4K_{12}(\lambda_1+\lambda_2)+4\lambda_1\lambda_2\\
&=-(\lambda_1+\lambda_2)-\tfrac{1}{3}+2W_{12}-4K_{12}(\lambda_1+\lambda_2)
\end{align*}
in the barrier sense. where in the last line we used $\lambda_1=R_{11}=0$
because $R=1$, $\nabla R=2Ric(\nabla f)$ and $e_1=\frac{\nabla f}{|\nabla f|}$.
Integrating the above inequality over $D(a, b)$,
\begin{equation}\label{eigenscal}
\int_{D(a, b)}\Delta_f (\lambda_1+\lambda_2)
\leq \int_{D(a, b)}-(\lambda_1+\lambda_2)-\tfrac{1}{3}+2W_{12}-4K_{12}(\lambda_1+\lambda_2).
\end{equation}

Notice that on a $4$-dimensional simply connected shrinker with $R=1$, the zero
set $f^{-1}(0)$ is a $2$-dimensional simply connected closed manifold. Actually,
$f^{-1}(0)$ is the deformation contraction of $M$, and hence it is diffeomorphic
to $\mathbb{S}^2$. Because
$f(x)=\frac{1}{4}d(x, f^{-1}(0))^2$, $f=|\nabla f|^2$ and the exponential map is
a local diffeomorphism, it is easy to see that $f^{-1}(t)$ is diffeomorphic to
$\mathbb{S}^2\times \mathbb{S}^1$ when $t$ is small. Thus all level sets of $f$
are diffeomorphic to $\mathbb{S}^2\times \mathbb{S}^1$ since $f$ has no critical
point away from $f^{-1}(0)$. All the above discussion implies that $D(a, b)$ is
diffeomorphic to $S^2\times S^1\times (a, b)$, hence $\chi(D(a, b), \Sigma(a)\cup \Sigma(b))=0$.

In the following we shall estimate the right hand side of \eqref{eigenscal}.
For this purpose, we begin with the following two claims.

\vspace{0.5em}

\textbf{Claim 3}: For sufficiently large $a$ and $b$, we have
\[
\int_{D(a, b)}W_{12}\leq \frac{1}{6}Vol(D(a, b))+\frac{1}{a}\int_{\Sigma(a)} O(\lambda_2)-\frac{1}{b}\int_{\Sigma(b)} O(\lambda_2).
\]
We now prove this claim. By Lemma \ref{1/12lemma} and Gauss-Bonnet-Chern formula
with boundary \eqref{Gauss-Bonnet-Chern}, we have
\begin{align*}
&\int_{D(a, b)}W_{12}\\
&\le \left(\int_{D(a, b)}W_{12}^2\right)^\frac{1}{2} Vol(D(a, b))^\frac{1}{2}\\
&\le\left(\int_{D(a, b)}\tfrac{|W|^2}{12}\right)^\frac{1}{2} Vol(D(a, b))^\frac{1}{2}\\
&=\left(\tfrac{1}{6}\int_{D(a, b)}(|Ric|^2-\tfrac{1}{3}R^2)+\tfrac{2}{3}\int_{\Sigma(a)}\digamma -\tfrac{2}{3}\int_{\Sigma(b)}\digamma\right)^\frac{1}{2} Vol(D(a, b))^\frac{1}{2}\\
&=\left(\tfrac{1}{36}Vol(D(a, b))+\tfrac{2}{3}\int_{\Sigma(a)}\digamma -\tfrac{2}{3}\int_{\Sigma(b)}\digamma   \right)^\frac{1}{2}Vol(D(a, b))^\frac{1}{2}\\
&=\left\{\left(\tfrac{1}{36}Vol(D(a, b))\right)^\frac{1}{2}+\frac{\frac{2}{3}\int_{\Sigma(a)}\digamma -\frac{2}{3}\int_{\Sigma(b)}\digamma }{\left(\frac{1}{36}Vol(D(a, b)){+}\frac{2}{3}\int_{\Sigma(a)}\digamma {-}\frac{2}{3}\int_{\Sigma(b)}\digamma\right)^\frac{1}{2}{+}\left(\frac{1}{36}Vol(D(a, b))\right)^\frac{1}{2}}\right\}\\
&\quad\times Vol(D(a, b))^\frac{1}{2}\\
&=\frac{1}{6}Vol(D(a, b))+\frac{\left(\frac{2}{3}\int_{\Sigma(a)}\digamma -\frac{2}{3}\int_{\Sigma(b)}\digamma\right)\times Vol(D(a, b))^\frac{1}{2}}{\left(\frac{1}{36}Vol(D(a, b)){+}\frac{2}{3}\int_{\Sigma(a)}\digamma {-}\frac{2}{3}\int_{\Sigma(b)}\digamma\right)^\frac{1}{2}{+}\left(\frac{1}{36}Vol(D(a, b))\right)^\frac{1}{2}},
\end{align*}
where
\[
\digamma=\tfrac{1}{2}R H-R_{11}H-R_{kikj}L^{ij}+\tfrac{1}{3}H^3-H |L|^2+\tfrac{2}{3}tr L^3.
\]
We remark that in the first equality we used $\chi(D(a, b), \Sigma(a)\cup \Sigma(b))=0$;
in the second equality we used $2|Ric|^2=R=1$.

Now we want to estimate the above complicated formula. To achieve it, we need
to deal with the quantity $\digamma$. We start to estimate the term
$R_{kikj}L^{ij}$. Notice that
\begin{equation}\label{Lij}
L_{ij}=\langle \nabla_{e_i}\tfrac{\nabla f}{|\nabla f|}, e_j\rangle =\frac{\nabla^2 f(e_i, e_j)}{|\nabla f|}=\frac{\frac{1}{2}g_{ij}-R_{ij}}{|\nabla f|},
\end{equation}

\begin{equation}\label{Lijtr}
H=g^{ij}L_{ij}=\frac{\tfrac{3}{2}-R}{|\nabla f|}=\frac{1}{2|\nabla f|},
\end{equation}

\begin{equation}
\begin{aligned}\label{RijLij}
R_{kikj}L^{ij}&=(R_{ij}-R_{1i1j})\frac{\frac{1}{2}g_{ij}-R_{ij}}{|\nabla f|}\\
&=\left(\tfrac{1}{2}R-|Ric|^2-\tfrac{1}{2}R_{11}+R_{1i1j}R_{ij}\right)\tfrac{1}{|\nabla f|}\\
&=\frac{R_{1i1j}R_{ij}}{|\nabla f|},
\end{aligned}
\end{equation}
where we used $2|Ric|^2=R=1$ and $R_{11}=0$. On a shrinker
with constant scalar curvature, we have (see Proposition 1 in \cite{[PW]})
\begin{align}\label{changshuliangqulvzhongyaodedengshi}
\nabla_{\nabla f}Ric=Ric\circ(Ric-\tfrac{1}{2}g)+Rm(\nabla f, \cdot, \nabla f, \cdot),
\end{align}
which gives that
\begin{equation}
\begin{aligned}\label{RijklRij}
R_{1i1j}R_{ij}&=\tfrac{1}{|\nabla f|^2}\left(\nabla_{\nabla f}R_{ij}\cdot R_{ij} -\lambda_i(\lambda_i-\tfrac{1}{2})\lambda_i  \right)\\
&=\tfrac{1}{|\nabla f|^2}\left(\tfrac{1}{2} \nabla_{\nabla f}|Ric|^2-(\lambda_i-\tfrac{1}{2})^2\lambda_i+\tfrac{1}{4}R-\tfrac{1}{2}|Ric|^2\right)\\
&=-\tfrac{1}{|\nabla f|^2}(\lambda_i-\tfrac{1}{2})^2\lambda_i,
\end{aligned}
\end{equation}
where we also used  $2|Ric|^2=R=1$. We observe that
\begin{align*}
0&=\tfrac{1}{2}\Delta_f R\\
&=\sum_{i=1}^4\lambda_i(\tfrac{1}{2}-\lambda_i)\\
&=\sum_{i=3}^4\left((\lambda_i-\tfrac{1}{2})(\tfrac{1}{2}-\lambda_i)
+\tfrac{1}{2}(\tfrac{1}{2}-\lambda_i)\right)
+\lambda_2(\tfrac{1}{2}-\lambda_2)\\
&=-(\tfrac{1}{2}-\lambda_3)^2-(\tfrac{1}{2}-\lambda_4)^2
+\frac{1}{2}(1-\lambda_3-\lambda_4)+\lambda_2(\tfrac{1}{2}-\lambda_2)\\
&=-(\tfrac{1}{2}-\lambda_3)^2-(\tfrac{1}{2}-\lambda_4)^2+\lambda_2(1-\lambda_2),
\end{align*}
which gives that
\begin{equation}\label{eigenequa}
(\tfrac{1}{2}-\lambda_3)^2+(\tfrac{1}{2}-\lambda_4)^2=\lambda_2(1-\lambda_2).
\end{equation}
Substituting \eqref{RijklRij} and \eqref{eigenequa} into \eqref{RijLij}, we have
\begin{equation}\label{RLO}
R_{kikj}L^{ij}=\frac{R_{1i1j}R_{ij}}{|\nabla f|}
=\frac{1}{|\nabla f|^3}\sum_{i=2}^4(\lambda_i-\tfrac{1}{2})^2\lambda_i
=\frac{O(\lambda_2)}{|\nabla f|^3}.
\end{equation}

Then we estimate $\frac{1}{3}H^3-H |L|^2 +\frac{2}{3}tr L^3$.
By \eqref{Lij}, \eqref{Lijtr} and \eqref{eigenequa}, we have
\begin{equation}
\begin{aligned}\label{HLH}
\tfrac{1}{3}H^3-H |L|^2+\tfrac{2}{3}tr L^3
&=\frac{1}{24|\nabla f|^3}-\frac{1}{2|\nabla f|}\sum_{i=2}^4\left(\frac{\frac{1}{2}-\lambda_i}{|\nabla f|}\right)^2+\tfrac{2}{3}\sum_{i=2}^4\left(\frac{\frac{1}{2}-\lambda_i}{|\nabla f|}\right)^3\\
&=\frac{1}{|\nabla f|^3}\Big[\tfrac{1}{24}-\tfrac{1}{2}(\tfrac{1}{2}-\lambda_2)^2-\tfrac{1}{2}(\tfrac{1}{2}-\lambda_3)^2
-\tfrac{1}{2}(\tfrac{1}{2}-\lambda_4)^2\\ &\qquad\qquad\quad+\tfrac{2}{3}(\tfrac{1}{2}-\lambda_2)^3+\tfrac{2}{3}(\tfrac{1}{2}-\lambda_3)^3
+\tfrac{2}{3}(\tfrac{1}{2}-\lambda_4)^3\Big]\\
&=\frac{1}{|\nabla f|^3}\Big[\tfrac{1}{24}-\tfrac{1}{2}\cdot\tfrac{1}{4}+ \tfrac{2}{3}\cdot\tfrac{1}{8}+O(\lambda_2)\Big]\\
&=\frac{O(\lambda_2)}{|\nabla f|^3},
\end{aligned}
\end{equation}
where we used $(\tfrac{1}{2}-\lambda_3)^3\le(\tfrac{1}{2}-\lambda_3)^2$
and $(\tfrac{1}{2}-\lambda_4)^3\le(\tfrac{1}{2}-\lambda_4)^2$.
Combining \eqref{Lijtr}, \eqref{RLO} and \eqref{HLH}, we get
\[
\digamma=\frac{1}{4|\nabla f|}+\frac{O(\lambda_2)}{|\nabla f|^3}.
\]
Therefore,
\begin{align*}
\tfrac{2}{3}\int_{\Sigma(a)}\digamma -\tfrac{2}{3}\int_{\Sigma(b)}\digamma
&=\tfrac{2}{3}Vol(\Sigma(a))\tfrac{1}{4|\nabla f|}-\tfrac{2}{3}Vol(\Sigma(b))\tfrac{1}{4|\nabla f|}\\
&\quad +\tfrac{1}{a}\int_{\Sigma(a)} O(\lambda_2)-\tfrac{1}{b}\int_{\Sigma(b)} O(\lambda_2)\\
&=\tfrac{1}{a}\int_{\Sigma(a)} O(\lambda_2)-\tfrac{1}{b}\int_{\Sigma(b)} O(\lambda_2).
\end{align*}
We finally obtain
\begin{align*}
\int_{D(a, b)}W_{12}&=\tfrac{1}{6}Vol(D(a, b))\\
&\quad+\frac{\left(\frac{1}{a}\int_{\Sigma(a)} O(\lambda_2)-\frac{1}{b}\int_{\Sigma(b)} O(\lambda_2)\right) \times Vol(D(a, b))^\frac{1}{2}}{\left(\frac{1}{36}Vol(D(a, b))+\frac{1}{a}\int_{\Sigma(a)} O(\lambda_2)-\frac{1}{b}\int_{\Sigma(b)} O(\lambda_2)   \right)^\frac{1}{2}+\left(\frac{1}{36}Vol(D(a, b))\right)^\frac{1}{2}}\\
&\leq \tfrac{1}{6}Vol(D(a, b))+\tfrac{1}{a}\int_{\Sigma(a)} O(\lambda_2)-\tfrac{1}{b}\int_{\Sigma(b)} O(\lambda_2).
\end{align*}
Therefore Claim 3 follows.

\vspace{0.5em}

\textbf{Claim 4}: $K_{12}\to 0$  uniformly at infinity.

This claim can be easily derived from \eqref{changshuliangqulvzhongyaodedengshi}.
Indeed, from \eqref{changshuliangqulvzhongyaodedengshi}
we get that
\begin{align*}
2 K_{12}=\frac{1}{f}\left(2\nabla_{\nabla f} Ric(e_2, e_2)+\lambda_2-2\lambda_2^2\right).
\end{align*}
Since $\nabla_{\nabla f} Ric$ and $\lambda_2$ tends to zero uniformly at infinity,
then
\[
K_{12}=\frac{o(1)}{f},
\]
which gives $K_{12}\to 0$ uniformly at infinity and Claim 4 follows.

\

Now plugging Claims 3 and 4 into \eqref{eigenscal}, we have
\begin{align*}
&\int_{D(a, b)}\Delta_f (\lambda_1+\lambda_2)\\
&\leq \int_{D(a, b)}-0.9(\lambda_1+\lambda_2)-\frac{1}{3}+2W_{12}\\
&\leq \int_{D(a, b)}-0.9(\lambda_1+\lambda_2) -\frac{1}{3}Vol(D(a, b))+\frac{2}{6}Vol(D(a, b)+\frac{1}{a}\int_{\Sigma(a)} O(\lambda_2)-\frac{1}{b}\int_{\Sigma(b)} O(\lambda_2)\\
&=\int_{D(a, b)}-0.9(\lambda_1+\lambda_2) +\frac{1}{a}\int_{\Sigma(a)} O(\lambda_2)-\frac{1}{b}\int_{\Sigma(b)} O(\lambda_2),
\end{align*}
Letting $b\rightarrow\infty$, we hence get
\begin{align}\label{diyigebudengshi}
\lim_{b\rightarrow\infty}\int_{D(a, b)}\Delta_f (\lambda_1+\lambda_2) \leq\int_{M\setminus D(a)}-0.9(\lambda_1+\lambda_2) +\frac{1}{a}\int_{\Sigma(a)} O(\lambda_2).
\end{align}

In the following, we shall estimate the left hand side of \eqref{diyigebudengshi}.
Since $\lambda_1+\lambda_2$ is continuous and Lipschitz, by Remark \ref{inquiinte}
and Proposition \ref{keyprop3} in Appendix, we have
\begin{equation}
\begin{aligned}\label{Deltaf}
\int_{D(a, b)}\Delta_f (\lambda_1+\lambda_2)
&=\int_{\Sigma(b)}\langle \nabla(\lambda_1+\lambda_2), \tfrac{\nabla f}{|\nabla f|}\rangle -\int_{\Sigma(a)}\langle \nabla(\lambda_1+\lambda_2), \tfrac{\nabla f}{|\nabla f|}\rangle\\
&\quad-\int_{\Sigma(b)}(\lambda_1+\lambda_2)|\nabla f|+\int_{\Sigma(a)}(\lambda_1+\lambda_2)|\nabla f|+\int_{D(a, b)}(2-R)(\lambda_1+\lambda_2)\\
&=\int_{\Sigma(b)}\langle \nabla(\lambda_1+\lambda_2), \tfrac{\nabla f}{|\nabla f|}\rangle -\int_{\Sigma(a)}\langle \nabla(\lambda_1+\lambda_2), \tfrac{\nabla f}{|\nabla f|}\rangle\\
&\quad-\int_{\Sigma(b)}(\lambda_1+\lambda_2)|\nabla f|+\int_{\Sigma(a)}(\lambda_1+\lambda_2)|\nabla f|+\int_{D(a, b)}(\lambda_1+\lambda_2)
\end{aligned}
\end{equation}
holds for almost everywhere sufficiently large $a$ and $b$,
where we used $R=1$ in the last equality. In the following,
we will estimate the right hand side of \eqref{Deltaf}.

On one hand, since $(M^4, g, q_i)$ converges to $\mathbb{R}\times (\mathbb{R}\times \mathbb{S}^2)/\Gamma$, it is easy to derive that $\Delta Ric$ and $Ric-2Rm*Ric$ tends to zero uniformly at infinity. Then we get that $\nabla_{\nabla f}Ric$ tends to zero uniformly. Thus,
\begin{equation}\label{gradlam}
\left|\int_{\Sigma(b)}\langle \nabla(\lambda_1+\lambda_2), \tfrac{\nabla f}{|\nabla f|}\rangle\right|
\leq |\nabla_{\nabla f}Ric|\cdot \frac{1}{|\nabla f|}\cdot Vol(\Sigma(b))\rightarrow 0
\end{equation}
as $b\rightarrow \infty$.

On the other hand, we shall show that there exists $b_i\rightarrow \infty$ such that
\begin{equation}\label{siglam}
\int_{\Sigma(b_i)}(\lambda_1+\lambda_2)|\nabla f|\rightarrow 0
\end{equation}
as $i\rightarrow \infty$. To prove \eqref{siglam}, we observe that $\lambda_1+\lambda_2\rightarrow 0$ and
 $K_{12}\rightarrow 0$ at infinity. Also, $Ric\ge 0$ and $2|Ric|^2=R=1$. Using these information,
\eqref{kaishide} can be reduced to
\begin{equation}
\begin{aligned}\label{tezhengzhidehe}
\Delta_f(\lambda_1+\lambda_2)&\le(\lambda_1+\lambda_2)
-(\lambda_1+\lambda_2)[1-(\lambda_1+\lambda_2)]
+2K_{12}[1-2(\lambda_1+\lambda_2)]\\
&\quad+[1-(\lambda_1+\lambda_2)]^2-2[\tfrac 12-(\lambda^2_1+\lambda^2_2)]\\
&=-(\lambda_1+\lambda_2)+2 K_{12}\\
&\quad+[-(\lambda_1+\lambda_2)+2(\lambda_1+\lambda_2)^2+2(\lambda^2_1+\lambda^2_2)-4K_{12}(\lambda_1+\lambda_2)]\\
&\le-(\lambda_1+\lambda_2)+2 K_{12}
\end{aligned}
\end{equation}
in the barrier sense, on $M\setminus D(a)$ for some $a>0$.  Arguing similarly as in the proof of Theorem \ref{clasfi},
we can obtain a sequence $c_i\rightarrow \infty$ such that
\[
\int_{M \setminus D(c_i)}\Delta_f(\lambda_1+\lambda_2) e^{-f}
=-\int_{\Sigma(c_i)}\langle \nabla(\lambda_1+\lambda_2), \tfrac{\nabla f}{|\nabla f|}\rangle e^{-f}\geq 0.
\]
Integrating \eqref{tezhengzhidehe} over $M\setminus D(c_i)$ with respect to the weighted
measure $e^{-f}dv$ and using the above formula, we have
\begin{align*}
\int_{M \setminus D(c_i)}(\lambda_1+\lambda_2) e^{-f}&\leq \int_{M \setminus D(c_i)} 2K_{12} e^{-f}-\int_{M \setminus D(c_i)}\Delta_f   (\lambda_1+\lambda_2) e^{-f}\\
&\leq \int_{M \setminus D(c_i)} 2K_{12} e^{-f}.
\end{align*}
By the coarea formula to the above inequality, we have
\begin{align*}
\int_{c_i}^\infty \int_{\Sigma(s)}\frac{\lambda_1+\lambda_2}{|\nabla f|}e^{-f}\leq \int_{c_i}^\infty \int_{\Sigma(s)}\frac{2 K_{12}}{|\nabla f|}e^{-f},
\end{align*}
which implies that there exist $b_i\rightarrow\infty$ such that
\begin{align*}
\int_{\Sigma(b_i)}\frac{\lambda_1+\lambda_2}{|\nabla f|}e^{-f}\leq  \int_{\Sigma(b_i)}\frac{2 K_{12}}{|\nabla f|}e^{-f},
\end{align*}
namely,
\begin{align*}
\int_{\Sigma(b_i)}\frac{\lambda_1+\lambda_2}{|\nabla f|}\leq  \int_{\Sigma(b_i)}\frac{2 K_{12}}{|\nabla f|}.
\end{align*}
By multiplying $|\nabla f|^2$ (where $|\nabla f|^2=f-1$) on $\Sigma(b_i)$ on
both sides of the above formula, the above inequality is equivalent to
\begin{align*}
\int_{\Sigma(b_i)}(\lambda_1+\lambda_2)|\nabla f|\leq \int_{\Sigma(b_i)} 2K_{12}|\nabla f|.
\end{align*}
From the argument of the above Claim 4, we know that $K_{12}=\frac{o(1)}{f}$. Thus,
\begin{align*}
\int_{\Sigma(b_i)}(\lambda_1+\lambda_2)|\nabla f|&\leq \frac{o(1)}{b_i}\sqrt{b_i}Vol(\Sigma(b_i))\\
&\leq \frac{o(1)}{b_i}\sqrt{b_i}\cdot\sqrt{b_i}\rightarrow 0
\end{align*}
as $i\rightarrow\infty$. Therefore \eqref{siglam} follows.

Besides, we can also argue as Theorem \ref{clasfi} to find $a$ large enough such that
\begin{equation}\label{ineteineq}
-\int_{\Sigma(a)}\langle \nabla(\lambda_1+\lambda_2), \tfrac{\nabla f}{|\nabla f|}\rangle\geq 0.
\end{equation}

Now letting $b_i\rightarrow\infty$ in \eqref{Deltaf} and using
\eqref{gradlam}, \eqref{siglam} and  \eqref{ineteineq}, we finally get
\[
\lim_{b_i\rightarrow\infty}\int_{D(a, b_i)}\Delta_f (\lambda_1+\lambda_2)
\geq \int_{\Sigma(a)}(\lambda_1+\lambda_2)|\nabla f|+\int_{M\setminus D(a)}(\lambda_1+\lambda_2).
\]
Combining this  with \eqref{diyigebudengshi}, we have
\[
\int_{\Sigma(a)}(\lambda_1+\lambda_2)|\nabla f|+\int_{M\setminus D(a)}(\lambda_1+\lambda_2) \leq\int_{M\setminus D(a)}-0.9(\lambda_1+\lambda_2) +\tfrac{1}{a}\int_{\Sigma(a)} O(\lambda_2).
\]
Since the above inequality holds for sufficiently large $a$, this actually implies
\[
\lambda_1+\lambda_2=0
\]
on $M\setminus D(a)$. This implies that
\[
\lambda_3=\lambda_4\equiv \tfrac{1}{2}
\]
on $M\setminus D(a)$. So the function
\[
G:=tr(Ric^3)-\tfrac{1}{2}|Ric|^2=0
\]
on $M\setminus D(a)$. Since $G$ is an analytic function, we hence obtain that
$G\equiv 0$ on $M$. On the other hand, we observe that
\begin{align*}
G&=tr(Ric^3)-|Ric|^2+\tfrac{1}{4}R\\
&=\sum_{i=1}^4(\lambda_i-\tfrac{1}{2})^2 \lambda_i,
\end{align*}
where we used $R=1$ and $0=\Delta_f R=R-2|Ric|^2$. Thus we get
$\lambda_1=\lambda_2\equiv 0$ and $\lambda_3=\lambda_4 \equiv \frac{1}{2}$
due to $Ric\geq 0$ and the continuity of $\lambda_1+\lambda_2$.
That is to say, the Ricci curvature has constant rank $2$. Therefore
the conclusion follows by \cite{[FR]}.
\end{proof}

%%%%%%%%%%%%%%%%%%%%%%%%%%%%%%%%%%%%%%%%%%%%%%%%%%%%%%%%%%%%%%%%%%%%%%%%%%%%%%%%%%%%%%%%%%%%%%%%%%%%%%%%%%%%%%%%%%%%%%%%%%%%%%%%%%%%%%%%%%%%%%

\section{Appendix}
In this section, we give some technical propositions about the integration inequality,
which roughly say that classical divergence theorems still hold in the barrier setting.
These propositions are often used in the proof of our theorems.

Let $(M,g)$ be a complete manifold, $\nabla$ be the gradient operator on $M$,
$\Delta$ be the Laplacian on $M$, $\Omega$ be an open set of $M$ and
$F\in C^\infty(M)$. The $F$-Laplacian is defined by
\[
\Delta_F:=\Delta-\nabla F\cdot\nabla.
\]
We say that a continuous function $u\in C(\Omega)$ satisfies $\Delta_F u\le h$
for some $h\in C(\Omega)$ \emph{in the barrier sense}, if for any fixed $x\in \Omega$,
there exists a smooth function $v$ defined in a neighborhood $U(x)$ of $x$,
such that $u(x)=v(x)$, $u(y)\le v(y)$ for any $y\in U(x)$ and
\[
\Delta_F v(x)\le h(x).
\]
We say that $u\in C(\Omega)$ satisfies $\Delta_F u\le h$ on $\Omega$
\emph{in the sense of distribution}, if
\[
\int_{\Omega} u\Delta_F\phi\, e^{-F}\leq \int_{\Omega} h \phi\, e^{-F}
\]
for any $\phi\geq 0$ with $\phi\in C_c^\infty(\Omega)$.

Using Ishii's theorem in \cite{Ishii}, we can show that
\begin{theorem}\label{bard}
If $u\in C(\Omega)$ satisfies $\Delta_F u\le h$ for some $h\in C(\Omega)$
in the barrier sense, then it satisfies $\Delta_F u\le h$ on $\Omega$ in
the sense of distribution.
\end{theorem}
\begin{proof}[Proof of Theorem \ref{bard}]
For $\Psi\in C^\infty(\Omega)$, it is easy to check that
$\Delta_F \Psi\le h$ is equivalent to the Schr\"odinger equation
\[
(\Delta-H)\overline{\Psi}\le e^{-\frac{F}{2}}h,
\]
where $\overline{\Psi}:=e^{-\frac{F}{2}}\Psi$ and $H:=\frac 14|\nabla F|^2-\frac 12\Delta F$.

In our setting, if $\Delta_F u\le h$ on $\Omega$ in the barrier sense,
then for any fixed $x\in \Omega$, there exists a smooth function $v$ defined
in a neighborhood $U(x)$ of $x$, such that $u(x)=v(x)$, $u(y)\le v(y)$
for any $y\in U(x)$ and
\[
\Delta_F v(x)\le h(x).
\]
Using the above equivalent viewpoint, we have that smooth function
$\overline{v}:=e^{-\frac{F}{2}}v$ satisfies
\[
(\Delta-H)\overline{v}\le e^{-\frac{F}{2}}h
\]
at $x\in \Omega$, where $H:=\frac 14|\nabla F|^2-\frac 12\Delta F$.
Set $\overline{u}:=e^{-\frac{F}{2}}u$. Then we have
\[
\overline{u}(x)=e^{-\frac{F(x)}{2}}u(x)=e^{-\frac{F(x)}{2}}v(x)=\overline{v}(x)
\]
and
\[
\overline{u}(y)=e^{-\frac{F(y)}{2}}u(y)\le e^{-\frac{F(y)}{2}}v(y)=\overline{v}(y)
\]
for any $y\in U(x)$. This indicates that
\[
(\Delta-H)\overline{u}\le e^{-\frac{F}{2}}h
\]
on $\Omega$ in the barrier sense. By Ishii's criterion (see Theorem 1 in \cite{Ishii}),
we conclude that the above equation holds on $\Omega$ in the sense of distribution.
That is,
\[
\int_{\Omega} (\overline{u}\Delta \phi-H\overline{u}\phi)\le\int_{\Omega} e^{-\frac{F}{2}} h
\]
for any $\phi\geq 0$ with $\phi\in C_c^\infty(\Omega)$.
Since $\overline{u}=e^{-\frac{f}{2}}u$, then
\[
\int_{\Omega}(u\Delta \phi-Hu\phi)e^{-\frac{F}{2}}\le\int_{\Omega} h\phi e^{-\frac{F}{2}}
\]
for any $\phi\geq 0$ with $\phi\in C_c^\infty(\Omega)$. Finally,
letting $\psi=e^{\frac{F}{2}}\phi$, the above equation is reduced
to
\[
\int_{\Omega} u\Delta_F \psi e^{-F}\le \int_{\Omega} h \psi e^{-F}
\]
for any $\psi\geq 0$ with $\psi\in C_c^\infty(\Omega)$. The conclusion follows.
\end{proof}

Next we shall apply Theorem \ref{bard} to discuss some special cases for our
theorems in the preceding sections. We use the same notations as before. On
a $4$-dimensional shrinker $(M,g,f)$, let $u\in (M,g,f)$ be a continuous
and Lipschitz function and let $h\in (M,g,f)$ be a continuous integrable function
with respect to measure $e^{-f}dv$. Here $f$ is uniformly equivalent to the
distance function squared (see Theorem \ref{pote}). For a sufficiently large
$r$, $a$ and $b$ ($a<b$), set $D(r):=\{x\in M|f(x)\le r\}$,
$\Sigma(r):=\{x\in M|f(x)=r\}$ and $D(a,b):=\{x\in M|a\le f(x)\le b\}$.
Then Theorem \ref{bard} implies that
\begin{corollary}\label{bardis}
If $\Delta_f u\le h$ holds on $M\setminus D(r)$ in the barrier sense, then it
holds on $M\setminus D(r)$ in the sense of distribution.
\end{corollary}
%\begin{proof}[Proof of Lemma \ref{bardis}]
%If $\Delta_f u\le h$ holds on $M\setminus D(r)$ in the barrier sense,
%then for any $x\in M\setminus D(r)$, there
%exists a smooth function $v$ defined in a neighborhood $U(x)$ of $x$, such that
%$u(x)=v(x)$, $u(y)\le v(y)$ for $y\in U(x)$ and
%\[
%\Delta_f v(x)\le h(x)
%\]
%on $M\setminus D(r)$. Multiplying a nonnegative supported function
%$\phi\in C_c^\infty(M\setminus D(r), e^{-f}dv)$ in both sides of the above
%inequality and integrating it over $M\setminus D(r)$ with respect to measure
%$e^{-f}dv$, we have
%\[
%\int_{M\setminus D(r)}\phi \Delta_fv(x)\,e^{-f}\le \int_{M\setminus D(r)}\phi h(x)\,e^{-f}.
%\]
%Integration by parts in the left side and then using $v(x)=u(x)$ yields
%\begin{equation}\label{distribution}
%\int_{M\setminus D(r)} u\Delta_f \phi \,e^{-f}\le \int_{M\setminus D(r)} h \phi\, e^{-f}.
%\end{equation}
%The lemma follows.
%\end{proof}

Using Corollary \ref{bardis}, we have the following useful proposition.
\begin{proposition}\label{keyprop}
If a continuous and Lipschitz function $u$ satisfies $\Delta_f u\le h$
on $M\setminus D(r)$ in the barrier sense, then
\begin{align*}
-\int_{\Sigma(r)}\langle\nabla u, \tfrac{\nabla f}{|\nabla f|}\rangle e^{-f}\le \int_{M\setminus D(r)} h\, e^{-f}
\end{align*}
holds for almost everywhere sufficiently large $r$.
\end{proposition}

\begin{remark}\label{inquiinte}
From Proposition \ref{keyprop}, we would like to point out that the integral
quantity $\int_{M\setminus D(r)}\Delta_f u e^{-f}$ can be essentially viewed
as $-\int_{\Sigma(r)}\langle\nabla u,\tfrac{\nabla f}{|\nabla f|}\rangle e^{-f}$.
That is to say, the classical divergence theorem still holds in the barrier setting.
This viewpoint is also suitable to the following Propositions \ref{keyprop2}
and \ref{keyprop3}.
\end{remark}

\begin{proof}[Proof of Proposition \ref{keyprop}]
By Corollary \ref{bardis}, we have that
\begin{equation}\label{distribution}
\int_{M\setminus D(r)} u\Delta_f \phi \,e^{-f}\le \int_{M\setminus D(r)} h \phi\, e^{-f}
\end{equation}
for any $\phi\geq 0$ with $\phi\in C_c^\infty(M\setminus D(r), e^{-f}dv)$.
Using the integration by parts on the left hand side of \eqref{distribution},  we have
\begin{equation}\label{disinte}
-\int_{M\setminus D(r)} \nabla u\cdot\nabla \phi\,  e^{-f}\le \int_{M\setminus D(r)} h \phi\, e^{-f}
\end{equation}
for any $\phi\geq 0$ with $\phi\in C_c^\infty(M\setminus D(r), e^{-f}dv)$.
Since $u$ is continuous and Lipschitz, it is easy to see that the above integration
inequality also holds for $\phi\in W_0^{1, 2}(M\setminus D(r), e^{-f}dv)$.

Now we shall construct a cut-off function $\phi\in W_0^{1, 2}(M\setminus D(r), e^{-f}dv)$
to derive our conclusion. For $r>0$ and $\Delta r>0$, consider a continuous function $\psi(s):\mathbb{R}\to[0,1]$ satisfying $\psi(s)=0$ if $s\le r$, $\psi(s)=1$ if $s\ge r+\Delta r$
and $\psi(s)=\frac{s-r}{\Delta r}$ if $r\le s\le r+\Delta r$.
Define $\phi(x)=\psi(f(x))$, where $x\in M$. Then
\begin{align*}
\int_{M\setminus D(r)} \nabla u\cdot \nabla \phi\, e^{-f}
&=\frac{1}{\Delta r}\int_{D(r,r+\Delta r)}\nabla u\cdot \nabla f \,e^{-f}\\
&=\frac{1}{\Delta r}\int_r^{r+\Delta r}\int_{\Sigma(s)}\nabla u\cdot\tfrac{\nabla f}{|\nabla f|}e^{-f}.
\end{align*}
Letting $\Delta r\rightarrow 0+$ and combining \eqref{disinte}, by Lebesgue's differentiable theorem,
we obtain the desired result.
\end{proof}

Similar to the argument of Proposition \ref{keyprop}, we have the following
result in the unweighted setting. Here we only use Ishii's result
in \cite{Ishii} rather than Corollary \ref{bardis}.
\begin{proposition}\label{keyprop2}
If a continuous and Lipschitz function $u$ satisfies $\Delta_f u\le h$ on
$D(r)\subset(M,g,f)$ in the barrier sense, where $r$ is a large number,
then
\[
\int_{\Sigma(r)}\langle\nabla u,\tfrac{\nabla f}{|\nabla f|}\rangle
-\int_{\Sigma(r)}u|\nabla f|+\int_{D(r)}(2-R)u \le \int_{D(r)} h
\]
holds for almost everywhere sufficiently large $r$.
\end{proposition}
\begin{proof}[Proof of Proposition \ref{keyprop2}]
Since $\Delta_f u\le h$ holds on $D(r)$ in the barrier sense,
by Ishii's criterion (see Theorem 1 in \cite{Ishii}), we have
\begin{align*}
\int_{D(r)}u(x)\Delta\phi+ \int_{D(r)} \nabla f\cdot \nabla \phi u +\int_{D(r)}\Delta f \phi \leq \int_{D(r)}h \phi
\end{align*}
for any $\phi\geq 0$ with $\phi\in C_c^\infty(D(r),dv)$. Integration by parts yields
\begin{equation}\label{inteinequ}
-\int_{D(r)}\nabla u\cdot \nabla \phi + \int_{D(r)} \nabla f\cdot \nabla \phi u +\int_{D(r)}\Delta f \phi \leq \int_{D(r)}h \phi
\end{equation}
for any $\phi\geq 0$ with $\phi\in C_c^\infty(D(r),dv)$.
Since $u$ is continuous and Lipschitz, the above inequality also
holds for $\phi\in W_0^{1, 2}(D(r),dv)$.

For a large $r>0$ and a small $0<\Delta r<1/2$, consider a continuous function $\psi(s):\mathbb{R}\to[0,1]$ satisfying $\psi(s)=1$ if $s\le r-\Delta r$, $\psi(s)=0$ if $s\geq r$
and $\psi(s)=\frac{r-s}{\Delta r}$ if $r-\Delta r\le s\le r$.
Define $\phi(x)=\psi(f(x))$, where $x\in M$. Then
\begin{align*}
\int_{D(r)} \nabla u\cdot \nabla \phi
&=\frac{-1}{\Delta r}\int_{D(r-\Delta r,r)}\nabla u\cdot \nabla f \\
&=\frac{-1}{\Delta r}\int_{r-\Delta r}^{r}\int_{\Sigma(s)}\nabla u\cdot\tfrac{\nabla f}{|\nabla f|},
\end{align*}
and
\begin{align*}
\int_{D(r)} \nabla f\cdot \nabla \phi u
&=\frac{1}{\Delta r}\int_{D(r-\Delta r,r)}u |\nabla f|^2\\
&=\frac{1}{\Delta r}\int_{r-\Delta r}^{r}\int_{\Sigma(s)}u |\nabla f|.
\end{align*}
Substituting the above equalities into \eqref{inteinequ} and then letting $\Delta r\rightarrow 0^+$, we obtain
\begin{align*}
\int_{\Sigma(r)}\nabla u\cdot \tfrac{\nabla f}{|\nabla f|}
-\int_{\Sigma(r)}u|\nabla f|+\int_{D(r)}(2-R)u \le \int_{D(r)} h
\end{align*}
for almost everywhere sufficiently large $r$, where we used $\Delta f=2-R$ for shrinker
in the last equality. Thus the conclusion follows.
\end{proof}

In the end, we would like to point out that, by a similar argument of
Propositions \ref{keyprop} and \ref{keyprop2}, we can get the
following result.
\begin{proposition}\label{keyprop3}
If a continuous and Lipschitz function $u$ satisfies $\Delta_f u\le h$ on
$D(a,b)\subset(M,g,f)$ in the barrier sense, where $a$ and $b$ are
large numbers, then
\[
\int_{\Sigma(b)}\langle \nabla u, \tfrac{\nabla f}{|\nabla f|}\rangle -\int_{\Sigma(a)}\langle \nabla u, \tfrac{\nabla f}{|\nabla f|}\rangle
-\int_{\Sigma(b)}u|\nabla f|+\int_{\Sigma(a)}u|\nabla f|+\int_{D(a, b)}(2-R)u \le \int_{D(a,b)} h
\]
holds for almost everywhere sufficiently large $a$ and $b$.
\end{proposition}

%%%%%%%%%%%%%%%%%%%%%%%%%%%%%%%%%%%%%%%%%%%%%%%%%%%%%%%%%%%%%%%%%%%%%%%%%%%%%%%%%%%%%%%%%%%%%%%%%%%%%%%%%%%%%%%%%%%%%%%%%%%%%%%%%%%%%%%%%%%%%%

\textbf{Acknowledgements}.
The  authors would like to thank Professors Fengjiang Li,    Xiaolong Li, Yu Li,  Yongjia Zhang and Xi-Nan Ma
 for helpful discussions.
%%%%%%%%%%%%%%%%%%%%%%%%%%%%%%%%%%%%%%%%%%%%%%%%%%%%%%%%%%%%%%%%%%%%%%%%%%%%%%%%%%%%%%%%%%%%%%%%%%%%%%%%%%%%%%%%%%%%%%%%%%%%%%%%%%%%%%%%%%%%%%
%%%%%%%%%%%%%%%%%%%%%%%%%%%%%%%%%%%%%%%%%%%%%%%%%%%%%%%%%%%%%%%%%%%%%%%%%%%%%%%%%%%%%%%%%%%%%%%%%%%%%%%%%%%%%%%%%%%%%%%%%%%%%%%%%%%%%%%%%%%%%%
% ------------------------------------------------------------------------
\bibliographystyle{amsplain}

\end{document}